\documentclass[11pt]{amsart}
\usepackage{verbatim, latexsym, amssymb, amsmath}
\usepackage{epsfig}
\def\R{\mathbb R}
\def\N{\mathbb N}

\def\C{\mathbb C}

\def\H{\mathcal H}
\def\L{\mathcal L}
\def\x{\mathbf x}

\def\vol{\mathrm{vol}}

\def\area{\mathrm{area}}

\def\dist{\mathrm{dist}}

\newtheorem*{conj}{Conjecture}
\newtheorem*{thma}{Theorem \ref{main}}

\newtheorem{thm}{Theorem}[section]
\newtheorem{lemm}[thm]{Lemma}
\newtheorem{cor}[thm]{Corollary}
\newtheorem{prop}[thm]{Proposition}
\theoremstyle{remark}
\newtheorem{rmk}[thm]{Remark}

\theoremstyle{definition}
\newtheorem{defi}[thm]{Definition}

\title{Finite Time Singularities for Lagrangian Mean Curvature Flow}

\author{{Andr\'e Neves}}
\email{aneves@imperial.ac.uk}
\address{Department of Mathematics, Imperial College London, London SW7 2AZ}

\begin{document}

\begin{abstract}
	Given any embedded Lagrangian on a four dimensional compact Calabi-Yau, we find another Lagrangian in the same Hamiltonian isotopy class which develops a finite time singularity under mean curvature flow. This contradicts a  weaker version of the Thomas-Yau conjecture regarding long time existence and convergence of Lagrangian mean curvature flow.
\end{abstract}

\maketitle \markboth{Finite Time Singularities for Lagrangian Mean Curvature Flow} {Andr\'e Neves}

\section{Introduction}

One of the hardest open problems regarding the geometry of Calabi-Yau manifolds consists in determining when a given Lagrangian admits a minimal  Lagrangian (SLag) in its homology class or Hamiltonian isotopy class.  If such SLag exists then it is area-minimizing  and so one could  approach this problem by trying to minimize area among all Lagrangians in a given class. Schoen and Wolfson \cite{schoen} studied the minimization problem and showed that, when the real dimension is four, a Lagrangian minimizing area among all Lagrangians in a given class exists,  is smooth everywhere except finitely many points, but not necessarily a minimal surface. Later Wolfson \cite{wolfson} found a  Lagrangian sphere  with nontrivial homology on a given K3 surface for which the Lagrangian which minimizes area among all Lagrangians in this class is not an SLag and the surface which minimizes area  among all surfaces  in this class is not Lagrangian. This shows the subtle nature of the problem.

In another direction, Smoczyk \cite{smo0} observed that when the ambient manifold is K\"ahler-Einstein   the Lagrangian condition is preserved  by the gradient flow of the area functional (mean curvature flow) and so a natural question is whether one can produce SLag's using Lagrangian mean curvature flow. To that end,  R. P. Thomas and S.-T. Yau \cite[Section 7]{thomas} considered this question and proposed a notion of ``stability'' for Lagrangians in a given Calabi-Yau which we now describe.

Let $(M^{2n}, \omega, J,\Omega)$ be a compact Calabi-Yau with metric $g$ where $\Omega$ stands for the unit parallel section of the canonical bundle. Given $L \subseteq M$ Lagrangian, it is a simple exercise (\cite[Section 2]{thomas} for instance)  to see that
$$\Omega_L=e^{i\theta}\vol_L,$$
where $\vol_L$ denotes the volume form of $L$ and $\theta$ is a multivalued function defined on $L$ called the {\em Lagrangian angle}. All the Lagrangians considered will be {\em zero-Maslov class}, meaning that  $\theta$ can be lifted to a  well defined function on $L$. 
 Moreover if $L$ is zero-Maslov class with oscillation of Lagrangian angle less than $\pi$ (called {\em almost-calibrated}),  there is a natural choice for the phase of $\int_L \Omega$, which we denote by $\phi(L)$.  
  Finally, given any two Lagrangians  $L_1, L_2$  it is defined in \cite[Section 3]{thomas} a connected sum operation $L_1\# L_2$ (more involved then a simply topological connected sum). We refer the reader to  \cite[Section 3]{thomas} for the details.

\begin{defi}[Thomas-Yau Flow-Stability] Without loss of generality, suppose that the almost-calibrated Lagrangian $L$  has  $\phi(L)=0$.  Then $L$ is flow-stable if any of the following two happen.
\begin{itemize}
\item $L$  Hamiltonian isotopic to $L_1\# L_2$, where $L_1, L_2$ are two almost-calibrated  Lagrangians, implies that
$$[\phi(L_1), \phi(L_2)]\nsubseteq (\inf_L\theta, \sup_L \theta).$$
\item $L$  Hamiltonian isotopic to $L_1\# L_2$, where $L_1, L_2$ are almost-calibrated  Lagrangians, implies that
$$\area(L) \leq \int_{L_1} e^{-i\phi(L_1)}\Omega+\int_{L_2} e^{-i\phi(L_2)}\Omega.$$
\end{itemize}
\end{defi}

\begin{rmk}
 The notion of flow-stability defined in \cite[Section 7]{thomas} applies to a larger class than almost-calibrated Lagrangians. For simplicity, but also because the author (unfortunately) does not fully understand that larger class, we chose to restrict the definition to almost calibrated.
\end{rmk}

In \cite[Section 7]{thomas} it is then conjectured

\begin{conj}[Thomas-Yau Conjecture] Let $L$ be a flow-stable Lagrangian in a Calabi-Yau manifold. Then the Lagrangian mean curvature flow will exist for all time and converge to the unique SLag in its Hamiltonian isotopy class.
\end{conj}

The intuitive idea is that if  a singularity occurs it is because the flow is trying to decompose the Lagrangian  into ``simpler'' pieces and so, if we rule out this possibility,  no finite time singularities should occur. Unfortunately, their stability condition  is in general hard to check. For instance, the definition does not seem to be preserved by  Hamiltonian isotopies and so it is a highly  nontrivial statement the existence of Lagrangians which are flow-stable and not SLag. As a result, it becomes quite hard to disprove the conjecture because not many examples of flow-stable Lagrangians are known. For this reason there has been considerable interest in the following simplified version of the above conjecture (see \cite[Section 1.4]{wang}).
\begin{conj} 
Let M be Calabi-Yau and $\Sigma$ be a compact embedded Lagrangian submanifold with zero Maslov class, then the mean curvature flow of $\Sigma$ exists for all time and converges smoothly to a special Lagrangian submanifold in the Hamiltonian isotopy class of $\Sigma$.
\end{conj}
We remark that in \cite{wang} this conjecture is attributed to Thomas and Yau  but this is not correct because there is no mention of stability. For this reason, this conjecture, due to Mu-Tao Wang, is a weaker version of  Thomas-Yau conjecture.

Schoen and Wolfson \cite{schoen1} constructed solutions to Lagrangian mean curvature flow which become singular in finite time and where the initial condition is homologous to a SLag $\Sigma$. On the other hand,  we remark that the flow {\em does} distinguish between isotopy class and homology class. For instance, on a two dimensional torus, a curve $\gamma$  with a single self intersection which is homologous to a simple closed geodesic will develop a finite time singularity under curve shortening flow while if we make the more restrictive assumption that $\gamma$ is isotopic to a simple closed geodesic, Grayson's Theorem \cite{grayson} implies that the curve shortening flow will exist for all time and sequentially converge to a simple closed geodesic.

The purpose of his paper is to prove

\begin{thma}\label{main0} Let $M$ be a four real dimensional Calabi-Yau and $\Sigma$ an embedded Lagrangian. There is $L$ Hamiltonian isotopic to $\Sigma$ so that the Lagrangian mean curvature flow starting at $L$ develops a finite time singularity.
\end{thma}

\begin{rmk}
\begin{enumerate}
\item [1)] If we take $\Sigma$ to be a SLag, the theorem implies the second conjecture is false because $L$ is then zero-Maslov class.
\item[2)] Theorem A provides the first examples of  compact embedded Lagrangians  which are not homologically trivial  and for which mean curvature flow develops a finite time singularity. The main difficulty comes from the fact, due to the high codimension,  barrier arguments or maximum principle arguments do not seem to be as effective as in the codimension one case and thus new ideas are needed.
\item[3)] One way to picture $L$ is to imagine a very small Whitney sphere $N$ (Lagrangian sphere with a single transverse self-intersection at $p$ in $\Sigma$) and consider $L=\Sigma\# N$(see local picture in Figure \ref{modelo}). 
\item[4)] If $\Sigma$ is SLag, then for every $\varepsilon$ we can make the oscillation for the Lagrangian angle of $L$ lying  in $[-\varepsilon, \pi+\varepsilon]$. Thus $L$ is not almost-calibrated and so does not qualified to be flow-stable in the sense of Thomas-Yau.
\item[5)] It is a challenging open question whether or not one can find  $L$ Hamiltonian isotopic to a SLag with arbitrarily small oscillation of the Lagrangian angle such that  mean curvature flow develops finite time singularities.  More generally, it is a fascinating problem to state a Thomas-Yau type conjecture which would have easier to check hypothesis on the initial condition  and allows (or not) for the formation of a restricted type of singularities.
\end{enumerate}
\end{rmk}

{\bf Acknowledgements:} The author would like to thank Sigurd Angenent for his remarks regarding Section \ref{reg.eq} and Richard Thomas  and Dominic Joyce for their kindness in explaining the notion of stability for the flow. He would also like to express his gratitude to Felix Schulze for his comments on Lemma \ref{fatlemma} and to one of the referees for the extensive comments and explanations which improved greatly the exposition of this paper.

\section{Preliminaries and sketch of proof}

In this section we describe  the mains ideas that go into the proof of Theorem A but first we have to introduce some notation.

\subsection{Preliminaries}\label{prelim}
Fix  $(M,J,\omega,\Omega)$  a four dimensional Calabi-Yau manifold with Ricci flat metric $g$, complex structure $J$, K\"ahler form $\omega$, and calibration form $\Omega$.  For every $R$ set $g_R=R^2 g$ and consider $G$ to be an isometric embedding of $(M,g_{R})$ into some Euclidean plane $\R^n$.  $L$ denotes a smooth Lagrangian surface contained in $M$ and $(L_t)_{t\geq 0}$ a smooth solution to Lagrangian mean curvature flow with respect to one of the metrics $g_R$ (different $R$ simply change the time scale of the flow).  It is simple to recognize the existence of $F_t:L\longrightarrow \R^n$ so that the surfaces  $L_t=F_t(L)$ solve the equation
$$\frac{dF_t}{dt}(x)=H(F_t(x))=\bar H(F_t(x))+E(F_t(x), T_{F_t(x)}L_t),$$
where $H(F_t(x))$ stands for the mean curvature with respect to $g_R$, $\bar H(F_t(x))$ stands for the mean curvature with respect to the Euclidean metric and $E$ is some vector valued function defined on $\R^n\times G(2,n)$, with $G(2,n)$ being the set of $2$-planes in $\R^n$. The term $E$ can be made arbitrarily small by choosing $R$ sufficiently large. In order to avoid introducing unnecessary notation, we will not be explicit whether we are regarding $L_t$ being a submanifold of $M$ or $\R^n$.

Given any $(x_0,T)$ in $\R^n\times \R$, we consider the backwards heat kernel
$$\Phi(x_0,T)(x,t)=\frac{\exp\left(-\frac{|x-x_0|^2}{4(T-t)}\right)}{4\pi(T-t)}.$$
We need the following extension of Huisken's monotonicity \cite{huisken} formula which follows trivially from \cite[Formula (5.3)]{Wa1}.
\begin{lemm}[Huisken's monotonicity formula]\label{huisken1} Let $f_t$ be a smooth family of functions with compact support on $L_t$. Then
\begin{multline*}
	\frac{d}{dt}\int_{L_t} f_t\Phi(x_0,T)d\H^2=
	\int_{L_t}\left( \partial_t f_t-\Delta f_t \right)\Phi(x_0,T)d\H^2\\
	-\int_{L_t}\left|\bar H+\frac{E}{2}+\frac{(\x-x_0)^{\bot}}{2(T-t_0)}\right|^2\Phi(x_0,T)d\H^2+\int_{L_t} f_t\Phi(x_0,T)\frac{|E|^2}{4}d\H^2.
\end{multline*}
\end{lemm}

We  denote  
$$A(r_1, r_2)=\{x\in \R^{n}\,|\, r_1<|x|<r_2\},\quad B_{r}=A(-1,r),$$
and define the $C^{2,\alpha}$ norm of a surface $N$ at a point $x_0$ in $\R^n$ as in \cite[Section 2.5]{white}. This norm is scale invariant and, given an open set $U$,  the $C^{2,\alpha}(U)$ norm of $N$ denotes the supremum in $U$ of the pointwise $C^{2,\alpha}$ norms.  We say $\bar N$ is $\nu$-close in $C^{2,\alpha}$ to $N$ if there is an open set $U$ and a function $u:N\cap U\longrightarrow \R^n$ so that $\bar N=u(N\cap U)$ and the $C^{2,\alpha}$ norm of $u$ (with respect to the induced metric on $N$) is smaller than $\nu$.

\subsubsection{Definition of $N(\varepsilon,\underline R)$}\label{prelim.n}

Let $c_1,$ $c_2$, and $c_3$ be three half-lines in $\C$ so that $c_1$  is  the positive real axis and $c_2,$  $c_3$ are, respectively, the positive line segments spanned by $e^{i\theta_2}$ and $e^{i\theta_3}$, where $\pi/2<\theta_2<\theta_3<\pi.$ These curves generate three Lagrangian planes in $\R^4$ which we denote by $P_1, P_2,$ and $P_3$ respectively.
 Consider a curve $\gamma(\varepsilon):[0,+\infty)\longrightarrow \C$  such that (see Figure \ref{modelo})
 \begin{itemize}
\item $\gamma(\varepsilon)$ lies in the first and second quadrant and $\gamma(\varepsilon)^{-1}(0)=0$;
\item  $\gamma(\varepsilon)\cap A(3,\infty)=c_1^+\cap  A(3,\infty)$ and $\gamma(\varepsilon)\cap A(\varepsilon,1)=(c_1^+\cup c_2\cup c_3)\cap  A(\varepsilon,1)$; 
\item $\gamma(\varepsilon)\cap B_{1}$ has two connected components $\gamma_1$ and $\gamma_2$, where $\gamma_1$ connects $c_2$ to $c_1^+$ and $\gamma_2$ coincides with $c_3$; 
\item The Lagrangian angle of $\gamma_1$, $\arg \left(\gamma_1 \frac{d\gamma_1}{ds}\right)$,  has oscillation strictly smaller than $\pi/2$.
\end{itemize}
Set $\gamma(\varepsilon,\underline R)=\underline R\gamma(\varepsilon/\underline R)$. We define
\begin{equation}\label{basic.block.eq}
N(\varepsilon,\underline R)= \{(\gamma(\varepsilon,\underline R)(s)\cos \alpha,\gamma(\varepsilon,\underline R)(s)\sin \alpha)\,|\, s\geq 0, \alpha \in S^1\}.
\end{equation}

\begin{figure}
\centering {\epsfig{file=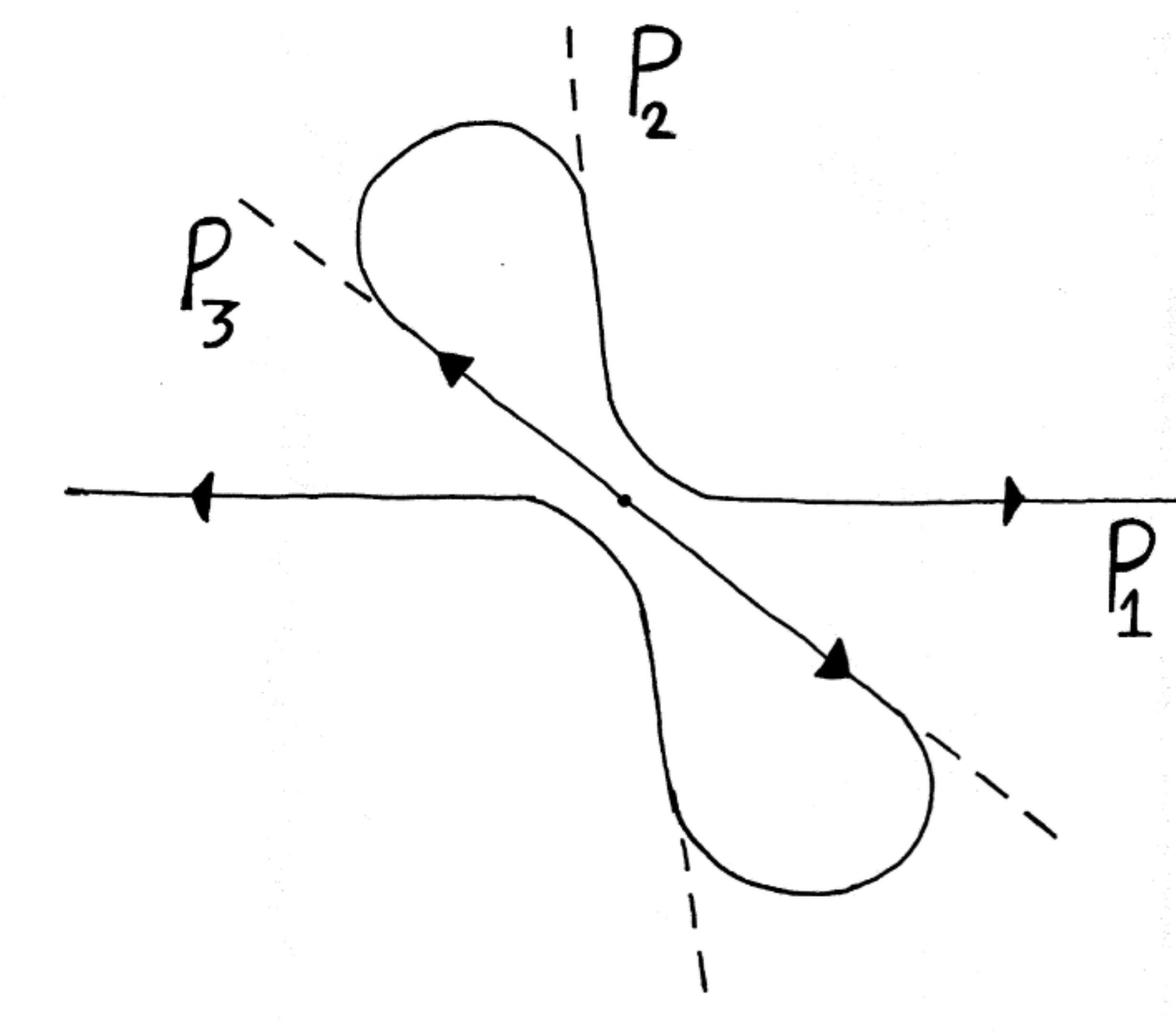, height=160pt}}\caption{Curve $\gamma(\varepsilon)\cup -\gamma(\varepsilon)$.}\label{modelo}
\end{figure}

We remark that one  can make the oscillation for the Lagrangian angle of $N(\varepsilon, \underline R)$ as close to $\pi$ as desired by choosing $\theta_2$ and $\theta_3$ very close to $\pi/2$.
\subsubsection{Definition of self-expander}\label{prelim.s}
A surface $\Sigma\subseteq \R^4$ is called a {\em self-expander} if $H=\frac{x^{\bot}}{2}$, which is equivalent to say that $\Sigma_t=\sqrt t \Sigma$ is a solution to mean curvature flow. We say that $\Sigma$ is {\em asymptotic} to a varifold $V$ if, when $t$ tends to zero,  $\Sigma_t$ converges in the Radon measure sense to $V$.  For instance, Anciaux \cite[Section 5]{anciaux} showed there is a unique curve  $\chi$ in $\C$ so that
\begin{equation}\label{Q}
\mathcal{S}=\{(\chi(s)\cos \alpha, \chi(s)\sin \alpha)\,|\, s\in \R,\alpha \in S^1\}
\end{equation} is a self-expander for Lagrangian mean curvature flow asymptotic to $P_1+P_2.$ 
\begin{figure}
\centering {\epsfig{file=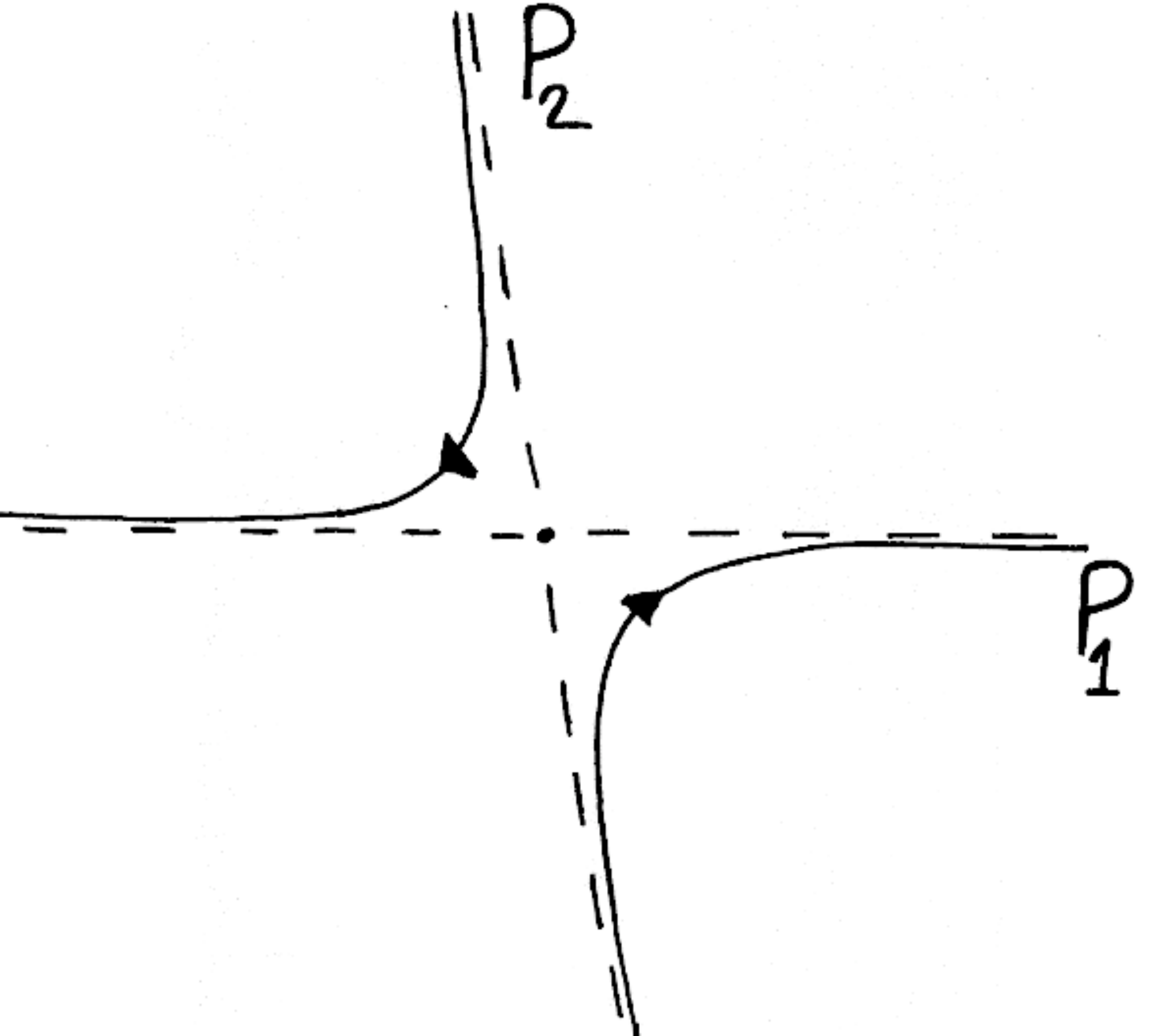, height=160pt}}\caption{Curve $\chi\cup -\chi$.}\label{modelo-se} 
\end{figure}

\subsection{Sketch of Proof}\label{sketch}

\begin{thma} Let $M$ be a four real dimensional Calabi-Yau and $\Sigma$ an embedded Lagrangian. There is $L$ Hamiltonian isotopic to $\Sigma$ so that the Lagrangian mean curvature flow starting at $L$ develops a finite time singularity.
\end{thma}
\begin{rmk}
The argument to prove Theorem \ref{main} has two main ideas. The first is to construct $L$ so that if the flow $(L_t)_{t\geq 0}$ exists smoothly, then $L_1$ and $L$ will be in  different Hamiltonian isotopy classes. Unfortunately this does not mean  the flow must become singular because Lagrangian mean curvature flow is not an {\em ambient} Hamiltonian isotopy. This is explained below in First Step and Second Step.

The second main idea is to note that $L_1$ is very close to a $SO(2)$-invariant Lagrangian $M_1$  which has the following property. The flow $(M_t)_{t\geq 1}$ develops a singularity at some time $T$ and  the Lagrangian angle will jump by $2\pi$ at instant $T$. Because the solution $(L_t)_{t\geq 1}$ will be ``nearby'' $(M_t)_{t\geq 1}$, this jump will  also occur on  $(L_t)_{t\geq 1}$  around time $T$ which means that it must have a singularity as well.

\end{rmk}
\begin{proof}[Sketch of proof]
It suffices to find a singular solution to Lagrangian mean curvature flow with respect to the metric $g_{R}=R^2 g$ for $R$ sufficiently large. Pick Darboux coordinates  defined on $B_{4\overline R}$ which send the origin into $p\in\Sigma$ so that $T_p\Sigma$ coincides with the real plane oriented positively and the pullback metric at the origin is Euclidean (we can increase $\overline R$ by making $R$ larger). The basic approach is to remove $\Sigma\cap B_{2\overline R}$ and replace it with $N(\varepsilon, \underline R)\cap B_{2\overline R}$. Denote the resulting Lagrangian by $L$ which, due to \cite[Theorem 1.1.A]{yasha}, we know to be  Hamiltonian isotopic to $\Sigma$.

Assume that the Lagrangian mean curvature flow $(L_t)_{t\geq 0}$ exists for all time. The goal is to get a contradiction when $\overline R,$  $\underline R$ are large enough and $\varepsilon$ is small enough.

\vspace{0.05in}

\noindent{\bf First step:} Because $L\cap A(1,2\overline R)$ consists of three planes which intersect transversely at the origin, we will use standard arguments based on White's Regularity Theorem \cite{white} and obtain estimates for the flow in a smaller annular region. Hence, we will conclude the existence of $R_1$ uniform so that $L_t\cap A(R_1, \overline R)$ is a small $C^{2,\alpha}$ perturbation of $L\cap A(R_1, \overline R)$ for all $1\leq t\leq2$ and the decomposition of $L_t\cap B_{\underline R}$  into two connected components $Q_{1,t}$, $Q_{2,t}$ for all $0\leq t\leq 2$, where  $Q_{2,0}=P_3\cap B_{\underline R}$. Moreover, we will also show that $Q_{2,t}$ is a small $C^{2,\alpha}$ perturbation of $P_3$ for all $1\leq t\leq 2$. This is done in Section \ref{geral} and the arguments are well-known among the experts. 

\vspace{0.05in}

\noindent{\bf Second step:}  In Section \ref{self-expanders} we show that $Q_{1,1}$ must be close to $\mathcal{S}$, the smooth self-expander asymptotic to $P_1$ and $P_2$ (see \eqref{Q} and Figure \ref{modelo-se}).  The geometric argument is that self-expanders act as attractors for the flow, i.e.,  because  $Q_{1,0}$   is very close to $P_1\cup P_2$ and $\sqrt t\mathcal{S}$ tends to $P_1+P_2$ when $t$ tends to zero,  then $Q_{1,t}$ must be very close to $\sqrt t \mathcal{S}$ for all $1\leq t\leq 2$. It is crucial for this part of the argument that $(Q_{1,t})_{0\leq t\leq 2}$ exists smoothly and that $P_1+P_2$ is not area-minimizing (see Theorem \ref{proximity.total} and Remark \ref{pro.total.rmk} for more details). This step is the first main idea of this paper.

From the first two steps it follows that  $L_1$ is very close to a Lagrangian $M_1$ generated by a curve $\sigma$ like the one in Figure \ref{modelo-3}. Because $Q_{1,0}$ is isotopic to $P_1\#P_2$  but $Q_{1,1}$ is isotopic to  $P_2\#P_1$ (in the notation of \cite{thomas}) we have that  $M_1$ is not Hamiltonian isotopic to $L$. Thus  it is not possible to connect the two by an ambient Hamiltonian isotopy. Nonetheless, as it was explained to the author by Paul Seidel, it is possible to connect them  by smooth Lagrangian immersions which are not rotationally symmetric nor embedded.  Unfortunately it is not known whether Lagrangian mean curvature flow is a Hamiltonian isotopy (only infinitesimal Hamiltonian deformation is known) and so there is no  topological obstruction to  go from $L$  to $L_1$ without singularities. 
\begin{figure}
\centering {\epsfig{file=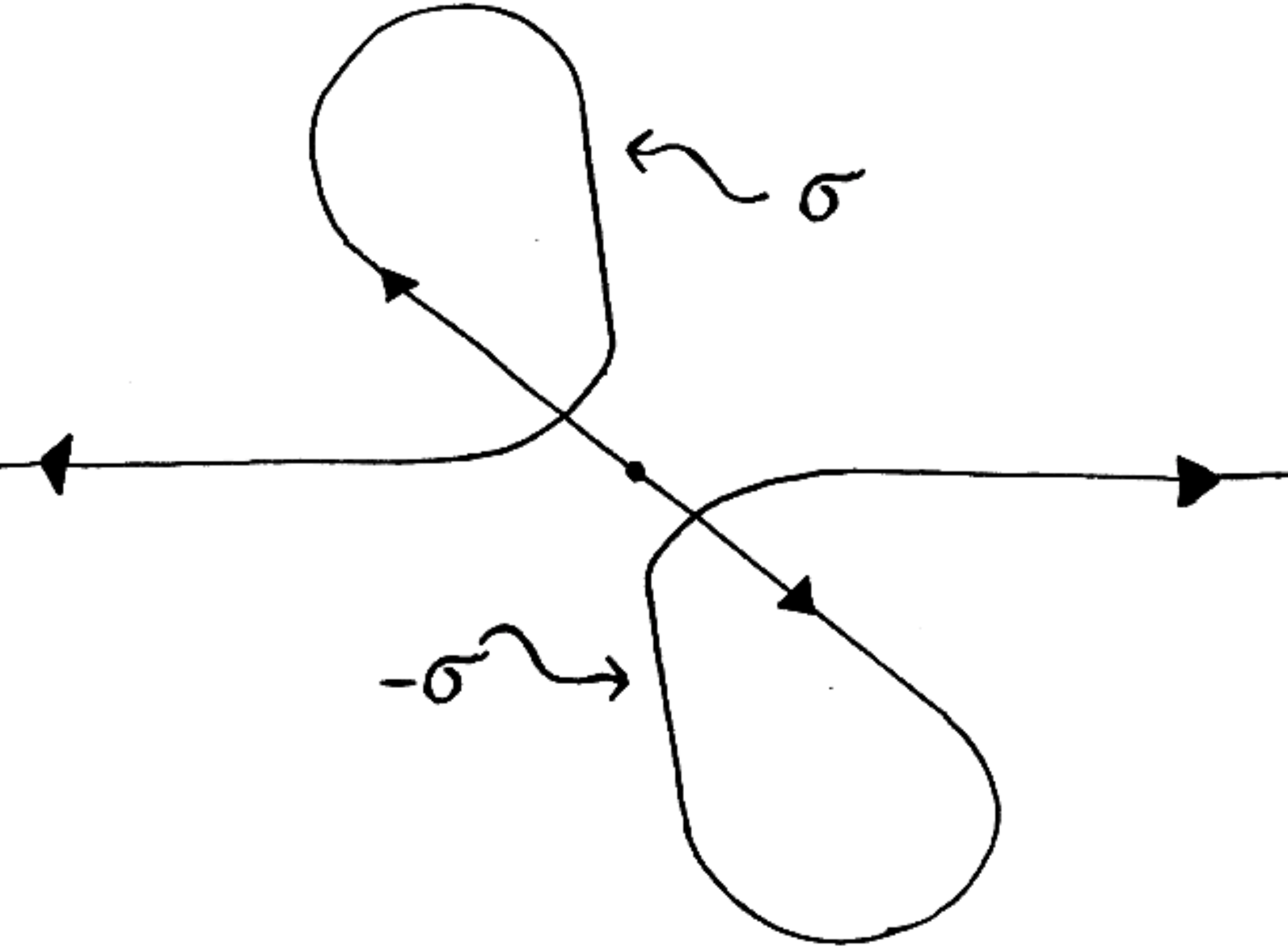, height=160pt}}\caption{Curve $\sigma\cup -\sigma$.}\label{modelo-3}
\end{figure}

Naturally we conjecture  that does not occur and that  $(L_t)_{0\leq t\leq 1}$ has a finite time singularity   which corresponds to the flow developing a ``neck-pinch'' in order to get rid of the non-compact ``Whitney Sphere'' $N(\varepsilon, \underline R)$ we glued to $\Sigma$. If the initial condition is simply $N(\varepsilon,\underline R)$ instead of $L$,  we showed in \cite[Section 4]{neves} that this conjecture is true but the arguments relied on the rotationally symmetric properties of $N(\varepsilon,\underline R)$ and thus cannot be extended to arbitrarily small perturbations like $L$. If this conjecture were true then the proof of Theorem \ref{main} would finish here. 

After several attempts, the author was unable to prove this conjecture and this lead us to the second main idea of this paper described below. Again we stress that, conjecturally, this case will never occur without going through ``earlier'' singularities.



\noindent{\bf Third step:}   Denote by $(M_t)_{t\geq 1}$ the evolution by mean curvature flow of $M_1$, the Lagrangian which corresponds to the curve $\sigma$. In Theorem \ref{sing.flow} we will show that $M_t$ is $SO(2)$-invariant and can be described by curves $\sigma_t$ which evolve the following way (see Figure \ref{modelo-4}). There is a singular time $T$ so that for all $t<T$ the curves $\sigma_t$ look like $\sigma$ but with a smaller enclosed loop. When $t=T$, this enclosed loop collapses and we have a singularity for the flow. For $t>T$, the curves $\sigma_t$ will become smooth and embedded.

We can now describe the second main idea of this paper (see Remark \ref{motiv.rmk}  and Corollary \ref{sing.equiv} for more details).
Because $\sigma_t$  ``loses'' a loop when $t$ passes through the singular time, winding number considerations will show that the Lagrangian angle of $M_t$ must suffer a discontinuity of $2\pi$. Standard arguments will show that, because $L_1$ is very close to $M_1$,  then $L_t$ will be very close to $M_t$  as well and so the Lagrangian angle of $L_t$ should also suffer a discontinuity of approximately $2\pi$ when $t$ passes through $T$. But this contradicts the fact that $(L_t)_{t\geq 0}$ exists smoothly.

\begin{figure}
\centering {\epsfig{file=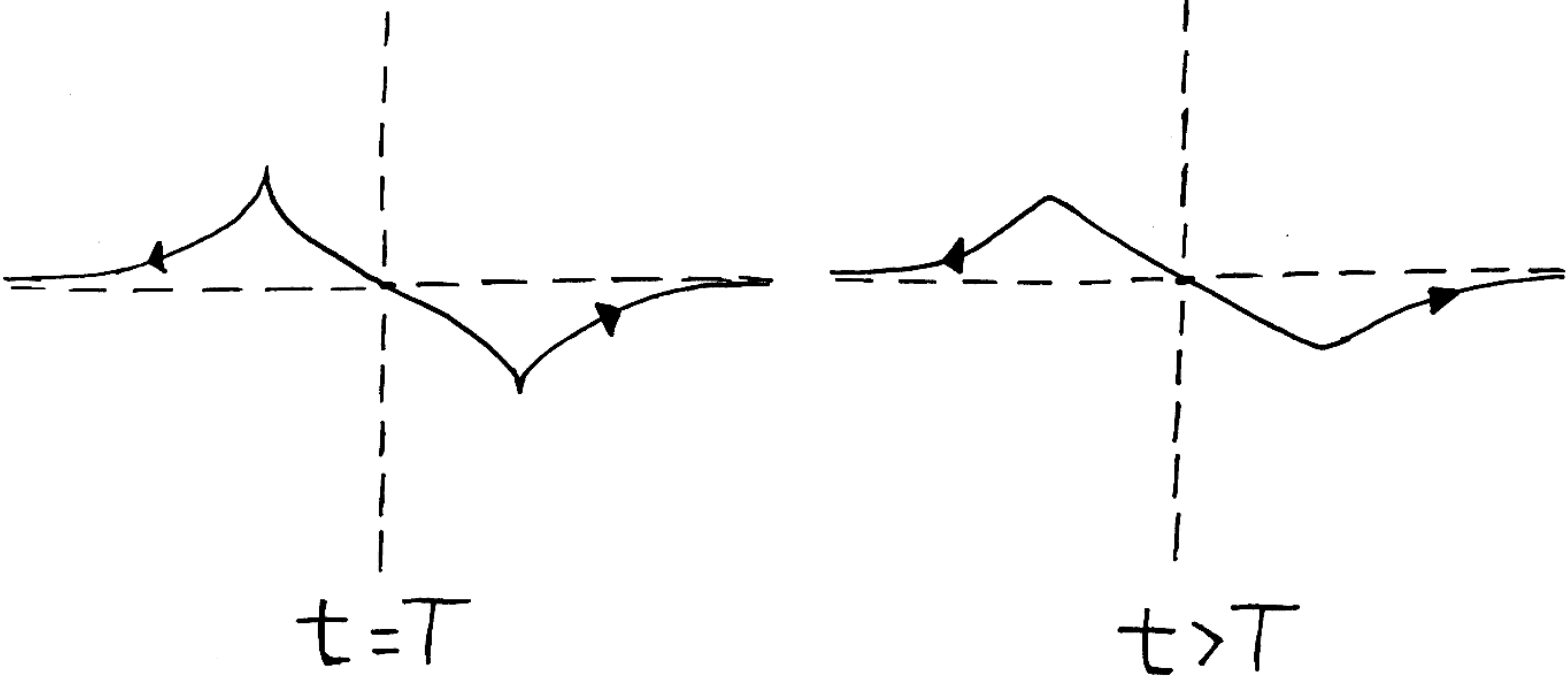, height=160pt}}\caption{Evolution of $\sigma_t.$}\label{modelo-4}
\end{figure}
\end{proof}

\subsection{Organization} The first step in the proof is done in Section \ref{geral} and it consists mostly of standard but slightly technical results, all of  which are well known. The second step is done in Section \ref{self-expanders} and the third step  is done in Section \ref{sta}. Finally, in Section \ref{ultima} the proof of Theorem \ref{main} is made rigorous and in the Appendix some basic results are collected.

Some parts of this paper are long and technical but can be skipped on a first reading.   Section \ref{geral} can be skipped and consulted only when necessary. On Section \ref{self-expanders} the reader can skip the proofs of Proposition \ref{monotone}, Proposition \ref{proximity_se} and read instead the outlines in Remark \ref{monotone.remark}, Remark \ref{proximity_se.rmk}. On Section \ref{sta} the reader can skip the proof of Theorem \ref{sing.flow}.

\section{First Step: General Results}\label{geral}
\subsection{Setup of Section \ref{geral}}\label{setup.geral} 
 \subsubsection{Hypothesis on ambient space}\label{ambient.space} 
 
 We assume the setting of Section \ref{prelim} and the existence of  a Darboux chart $$\phi:B_{4\overline R}\longrightarrow M,$$ meaning  $\phi^*\omega$ coincides with the standard symplectic form in $\R^4$ and  $\phi^*J$ and $\phi^*\Omega$ coincide, respectively,  with the standard complex structure and $dz_1\wedge dz_2$ at the origin.  Moreover, we assume that
 \begin{itemize}
  \item $\phi^*g_R$ is $ {1}/ {\overline R}$-close in $C^{3}$ to the Euclidean metric, 
  \item $G\circ \phi$ is $1/\overline R$-close in $C^{3}$ to the map that sends $x$ in $\R^4$ to $(x,0)$ in $\R^n$, 
  \item the $C^{0,\alpha}$ norm of $E$ (defined in Section \ref{prelim}) is smaller than $1/\overline R$, 
  \item and $G(M)\cap B_{4\overline R-1}\subseteq G\circ \phi(B_{4\overline R}).$ 
  \end{itemize}
  For the sake of simplicity, given any subset $B$ of $M$,  we  freely identify  $B$ with $\phi^{-1}(B)$ in $B_{4\overline R}$ or $G(B)$ in $\R^n$.
 
 \subsubsection{Hypothesis on Lagrangian $L$}\label{hyp.lagrangian}
 
We assume that  $L\subseteq M$ Lagrangian is such that
 \begin{equation}\label{defi2.geral}
 L\cap B_{2\overline R}=N(\varepsilon,\underline R) \cap B_{2\overline R} \mbox{ for some  } \overline R\geq 4\underline R, \end{equation}
 where $N(\varepsilon,\underline R)$ was defined in \eqref{basic.block.eq}. 
 Thus $L \cap B_{2\underline R}$ consists of two connected components $Q^1$ and $Q^2$, where 
 \begin{equation}\label{defi.geral}
 Q^1\setminus B_{\varepsilon}=(P_1 + P_2)\cap  A(\varepsilon, 2\underline R) \mbox{ and } Q^2=P_3\cap B_{2\underline R}.
 \end{equation}
 To be rigorous, one should use the notation $L_{\varepsilon,\underline R}$ for $L$. Nonetheless, for the sake of simplicity, we prefer the latter.  
 Finally we assume the existence of $K_0$ so that
\begin{itemize} 
\item $\area(L\cap B_r(x))\leq K_0 r^2\quad\mbox{for every } x\in M \mbox{ and }r\geq 0,$
\item the norm of second fundamental form of $M$ in $\R^n$ is bounded by $K_0,$  
\item $\sup_{Q^1} |\theta|\leq \pi/2-K_0^{-1}$ and  we can find   $\beta\in C^{\infty}(Q^1)$ such that $d \beta=\lambda$, and  
\begin{equation}\label{beta.connected}
|\beta(x)|\leq K_0(|x|^2+1)\mbox{ for all }x\in Q^1.
\end{equation}
\end{itemize}

\subsection{Main results}

We start with two basic lemmas and then state the two main theorems.   

\begin{lemm}\label{area.bounds} For all $\varepsilon$ small, $\underline R$ large, and $T_1>0$, there is  $D=D(T_1,K_0)$ so that
	$$\H^2(B_r(x)\cap L_t)\leq Dr^2 \mbox{ for all $x\in \R^n$, $r>0,$ and $0\leq t\leq T_1.$} $$
\end{lemm}
\begin{proof}
Assuming  a uniform bound on the second fundamental form of $M$ in $\R^n$, it is a standard fact that uniform area bounds for $L_t$ hold for all  $0\leq t\leq T_1$ (see for instance \cite[Lemma A.3]{neves}  if $g$ is the Euclidean metric. A general proof could be given along the same lines provided we use the modification of monotonicity formula given in Lemma \ref{huisken1}).
\end{proof}

\begin{lemm}\label{large}
For every $\delta$ small, $T_1>1$, and $\underline R>0$, there is $R=R(T_1,\delta, \underline R)$ so that, for every $1\leq t\leq T_1$,  $L_t$ is  $\delta$-close in $C^{2,\alpha}$ to the plane $P_1$  in the annular region $A(R,\overline R).$
\end{lemm}
\begin{proof}
Apply Lemma \ref{ap2} with $\nu$ being $\delta$ given in this lemma, $S=1$, and  $\kappa=1/T_1$. Because $L_0\cap A(3\underline R,2\overline R)=P_1\cap A(3\underline R,2\overline R)$, it is simple to see that conditions a), b), and c) of Lemma \ref{ap2} are satisfied for all $x_0\in L_0\cap A(R,\overline R)$ provided we  choose $R$ suitably large. Thus, the desired result follows from Lemma \ref{ap2} ii).
\end{proof}

The next theorem is one of the main results of this section. The proof will be  given at the end  of Section \ref{geral} and can be skipped on a first reading.

\begin{thm}\label{connected}
	  Fix $\nu$.  The constant $\Lambda_0$ mentioned below is universal.
	
	There are $\varepsilon_1$ and $R_1$,  depending on the planes $P_1$, $P_2,$ $P_3$, $K_0$, and $\nu$, such that  if $\varepsilon\leq \varepsilon_1$ and $\underline R\geq 2R_1$ in  \eqref{defi2.geral}, then
	\begin{itemize}
	\item[i)]
	for every $0\leq t\leq 2$, the $C^{2,\alpha}(A(R_1, \overline R))$ norm of $L_t$ is bounded by $\Lambda_0 t^{-1/2}$ and
	$$F_t(x)\in A(R_1, \overline R) \implies |F_s(x)-x|< \Lambda_0\sqrt s\mbox{ for all } 0\leq s\leq 2;$$
	\item[ii)] for every $1\leq t\leq 2$, 	$L_t\cap A(R_1, \overline R)$ is $\nu$-close in $C^{2,\alpha}$ to $L$. 
  \end{itemize}

	Moreover, setting
	$$Q_{1,t}=F_t(Q^1\cap B_{\underline R})\quad\mbox{and}\quad Q_{2,t}=F_t(Q^2\cap B_{\underline R}),$$
	we have  that
	\begin{itemize}
	\item[iii)] for every $0\leq t\leq 2$ 
	$$L_t\cap B_{\underline R-\Lambda_0}\subseteq Q_{1,t}\cup Q_{2,t}\subseteq B_{\underline R+\Lambda_0};$$
	\item[iv)]  for every $1\leq t\leq 2$, $Q_{2,t}$ is $\nu$-close in $C^{2,\alpha}(B_{R_1})$ to $P_3$.
	\end{itemize}
\end{thm}
	
\begin{rmk}\label{rmk.connected}
{\bf (1)}  We remark that Theorem \ref{connected} i) and iii) have no $\nu$ dependence in their statements and so could have been stated independently of Theorem ii) and iv).

{\bf (2)}  The content of Theorem \ref{connected} i) and ii) is that for all $\varepsilon$ small and $\underline R$ large we have good control of $L_t$ on an annular region $A(R_1,\overline R)$ for all $t\leq 2$. This is expected because, as  we explain next, for all $\epsilon$ small and $\underline R$ sufficiently large,  $L\cap A(1,2\overline R)$ has small $C^{2,\alpha}$-norm  and area ratios close to one. In the region $A(1,\underline R)$ this follows because, as  defined in   \eqref{defi2.geral}, 
$$L\cap A(1,\underline R)=(P_1\cup P_2\cup P_3)\cap A(1,\underline R).$$ 
 In the region $A(\underline R, 2\overline R)$ this follows because the   $C^{2,\alpha}$ norm  and the area ratios of $L\cap A(\underline R, 2\overline R)$ tend to zero as $\underline R$ tends to infinity. 

{\bf (3)} The content of Theorem \ref{connected} iii) is that  $L_t\cap B_{\underline R}$ has two distinct connected components for all $0\leq t\leq 2$ which we call $Q_{1,t}$ and $Q_{2,t}$. The idea is that initially $L\cap B_{\underline R}$ has two connected components and because we have control of the flow on the annulus $A(R_1,\underline R)$ due to Theorem \ref{connected} i), then no connected component in $L_t\cap B_{\underline R}$ can be ``lost'' or ``gained''.  Without the control on the annular region it is simple to construct examples where a solution to mean curvature flow in $B_1(0)$ consists initially of disjoint straight lines and at a later time is a single connected component.

{\bf (4)} Theorem \ref{connected} iv) is also expected because $Q_{2,0}$ initially is just a disc and we have good control on $\partial Q_{2,t}$ for all $0\leq t\leq 2$.

\end{rmk}

 The next theorem collects some important properties of $Q_{1,t}$. The proof will be  given at the end  of Section \ref{geral} and, because it is largely standard, can be skipped on a first reading.
 
 \begin{thm}\label{annulus.estimates}There are  $D_1$, $R_2,$ and $\varepsilon_2$ depending only on $K_0$ so that if  $\underline R\geq R_2$ and  $\varepsilon\leq \varepsilon_2$ in \eqref{defi2.geral}, then for every $0\leq t\leq 2$ the following properties hold.
 \begin{enumerate}
 
\item[i)]  $$\sup_{0\leq t\leq 2}\sup_{Q_{1,t}} |\theta_t|\leq \pi/2-1/(2K_0).$$
\item[ii)] \begin{equation*}\label{bb}
\H^2(\hat B_{r}(y))\geq D_1 r^2,
\end{equation*}
where $\hat B_{r}(y)$ denotes the intrinsic ball of radius $r$ in $Q_{1,t}$ centered at $y \in Q_{1,t}$ and $r<\dist(y,\partial Q_{1,t}).$
\item[iii)]All $Q_{1,t}$ are exact and one can choose $\beta_t \in C^{\infty}(Q_{1,t})$ with $$d\beta_t=\lambda=\sum_{i=1}^2 x_idy_i-y_idx_i$$ and
$$\frac{d}{dt}(\beta_t+2t\theta_t)=\Delta(\beta_t+2t\theta_t)+E_1,$$
where $E_1=\sum_{i=1}^2\nabla _{e_i}\lambda (e_i)$ and $\{e_1, e_2\}$ is an orthonormal basis for  $Q_{1,t}.$ 
\item[iv)] \begin{equation*}
	|\beta_t|(x)\leq D_1(|x|^2+1)\quad\mbox{for every }x\in Q_{1,t}.
	\end{equation*}
\item[v)] If $\mu=x_1y_2-x_2y_1$, then $$\frac{d\mu^2}{dt} \leq \Delta \mu^2-2|\nabla \mu|^2+E_2,$$
where $E_2=(|x|^3+1)O(1/\overline R).$

\end{enumerate}
 \end{thm}
 \begin{rmk}
{\bf (1)} We comment on Theorem \ref{annulus.estimates} i). Recall that we are assuming  $\sup_{Q^1} |\theta|\leq \pi/2-K_0^{-1}$, where  $Q^1$ is defined in \eqref{defi.geral}. Because $\theta_t$ evolves by the heat equation we have
 $$\sup_{0\leq t\leq 2}\sup_{Q_{1,t}}|\theta_t|\leq \max\left\{\sup_{Q_{1,0}}|\theta|,\sup_{0\leq t\leq 2}\sup_{\partial Q_{1,t}}|\theta_t|\right\}.$$
 Hence we need to control the Lagrangian angle along $\partial Q_{1,t}$ in order to obtain Theorem \ref{annulus.estimates} i). The idea  is to use the fact that $Q^1$ is  very ``flat'' near $\partial Q_{1,0}$ to show that $F_t(Q^1)$ is a small $C^1$ perturbation of $Q^1$ near $\partial Q_{1,0}$, which means the Lagrangian angle along $\partial Q_{1,t}$ will not change much.

{\bf (2)}  Theorem  \ref{annulus.estimates} ii) is a consequence of the fact that $Q_{1,t}$ is almost-calibrated.

{\bf (3)}  Theorem  \ref{annulus.estimates} iii) and v) are just derivations of evolution equations having into account the error term one obtains from the metric $g_R$ (defined in Section \ref{prelim}) not being Euclidean.

{\bf (4)} Theorem \ref{annulus.estimates} iv) gives the expected growth for $\beta_t$ on $Q_{1,t}$ and its proof is a simple technical matter.
 \end{rmk}

 \subsection{Abstract results}
 
 We derive some simple results which will be used to prove Theorem \ref{connected}, Theorem \ref{annulus.estimates}, and throughout the rest of the paper. They are presented in  a fairly general setting in order to be used in various circumstances. The proofs are based on White's Regularity Theorem and Huisken's monotonicity formula.
 
Let $E$ be a vector valued function defined on $\R^{n}\times G(2,n)$, $\Sigma$ a smooth surface possibly with boundary, and $F_t:\Sigma\longrightarrow \R^n$ a smooth solution to 
\begin{equation}\label{flow.abstract}
\frac{dF_t}{dt}(x)=H(F_t(x))+E(F_t(x), T_{F_t(x)}M_t),
\end{equation}
 where $M_t=F_t(\Sigma)$ and $F_0$ is the identity map. 
 
 In what follows $\Omega$ denotes a closed   set of $\R^{n}$ and  we use the notation
$$ \Omega(s)=\{x\in\R^{n}\,|\,\dist(x,\Omega)< s\}.$$

 We derive two lemmas which are well known among the experts. Denote $\bar E=\sup |E|_{0,\alpha}$ and let $\varepsilon_0$ be the constant given by White's Regularity Theorem \cite[Theorem 4.1]{white}.
\begin{lemm}\label{ap1}
Assume $T\leq 4$. There  is $\Lambda=\Lambda(\bar E, n)$ so that for every $s\geq 0$ if
\begin{itemize}
\item[a)] for all $0\leq t\leq 2T$, $y\in \Omega (s+2\Lambda\sqrt T)$, and $l\leq 2T$
$$\int_{M_t}\Phi(y,l)d\H^2\leq 1+\varepsilon_0;$$
\item[b)]for all $0\leq t\leq 2T$, $\partial M_t \cap \Omega (s+2\Lambda\sqrt T)=\emptyset;$  
\end{itemize}

then for every $0\leq t\leq T$  we have
\begin{itemize}
\item[i)] the $C^{2,\alpha}$ norm of $M_t$ on $\Omega(s+\Lambda\sqrt T)$ is bounded by $\Lambda/\sqrt{t}$;
\item[ii)]
$$F_t'(x)\in \Omega(s)\implies |F_t(x)-x|<\Lambda\sqrt t \mbox{ for all }0\leq t\leq T.$$
\end{itemize}
\end{lemm}
\begin{rmk} 
The content of Lemma \ref{ap2} is that if we know the Gaussian density ratios at a scale smaller than $2T$ in a region $U$ are all close to one and $\partial M_t$ lies outside $U$ for all $t\leq 2T$, then we have good control of $M_t$ for all $0\leq t\leq T$ on  a slightly smaller region. The proof is a simple consequence of White's Regularity Theorem.
\end{rmk}
\begin{proof}
	Assume  for all $0\leq t\leq 2T$, $y\in \Omega(s+(\Lambda+1)\sqrt T)$, and $l\leq 2T$
$$\int_{M_t}\Phi(y,l)d\H^2\leq 1+\varepsilon_0 $$
where $\Lambda\geq 1$ is a constant to be chosen later.  	From White's Regularity Theorem   \cite[Theorem 4.1]{white} there is $K_1=K_1(\bar E,n)$ so that the $C^{2,\alpha}$ norm of $M_t$ on $\Omega(s+\Lambda\sqrt T)$ is bounded by $K_1/\sqrt{t}$  and 
	$$\sup_{M_t\cap \Omega(s+\Lambda\sqrt T)}|A|^2\leq \frac{K_1}{t}$$
	for every $t\leq T$. Thus we obtain from \eqref{flow.abstract} 
	$$\left|\frac{dF_t}{dt}(x)\right|\leq \frac{K_1}{\sqrt t}+\bar E $$
	whenever $F_t(x)\in \Omega(s+\Lambda\sqrt T)$. Integrating the above inequality and using $T\leq 2$, we have  the existence of $K_2=K_2(\bar E,K_1)$ so that  if
	\begin{equation*}\label{implica}
	F_{t'}(x)\in\Omega(s+(\Lambda-K_2)\sqrt T)\mbox{ for some } 0\leq t'\leq T,
	\end{equation*}
	then
	$$
		F_t(x)\in \Omega(s+\Lambda\sqrt T) \quad\mbox{for every} \quad0\leq t\leq T
	$$ 
	and $$|F_t(x)-x|< K_2\sqrt t \quad\mbox{for every}\quad 0\leq t\leq T.$$
	Choose $\Lambda=\max\{K_1,K_2\}$. Then i) and  ii) follow at once. 
	\end{proof}

\begin{lemm}\label{ap2}
 For every $\nu$, $S,$ and $0<\kappa<1$, there is $\delta$, $R$ so that if $x_0\in M_0$ and
\begin{enumerate}
\item[a)]the $C^{2,\alpha}$ norm of $M_0$  in $B_{R\sqrt T}(x_0)$ and the $C^{0,\alpha}(\R^{n}\times G(2,M))$ norm of $E$ are smaller than $\delta/\sqrt T$;
\item[b)]$\H^2(M_0\cap B_r(x_0))\leq 7\pi r^2$ for all  $0\leq r\leq R\sqrt T$;
\item[c)]$\partial M_t\cap B_{R\sqrt T}(x_0)=\emptyset$ for all $0\leq t\leq T$;
\end{enumerate}
  then  the following hold:
\begin{itemize}
\item[i)] $$\int_{M_t}\Phi(y,l)d\H^2\leq 1+\varepsilon_0\quad\mbox{for all } y\in B_{(S+1)\sqrt T}(x_0), \mbox{ $t\leq T$, and }l\leq 2T;$$
\item[ii)] For every $\kappa T\leq t\leq T$ there is a function $$u_t:T_{x_0} M_0\cap B_{(S+1)\sqrt T}(x_0)\longrightarrow (T_{x_0} M_0)^{\bot}$$ with 
$$\sup_{T_{x_0} M_0\cap B_{(S+1)\sqrt T}} \left(|u_t| /\sqrt T+|\nabla u_t|+|\nabla^2 u_t|_{0,\alpha} \sqrt T\right)\leq \nu $$
and 
$$M_t\cap B_{S\sqrt T}(x_0)\subseteq \{u_t(x)+x,|\, x\in T_{x_0} M_0\cap B_{(S+1)\sqrt T}(x_0)\}.$$
\end{itemize}
\end{lemm}
\begin{rmk}
This lemma, roughly speaking,  says that  for every $S$, there is $R$ so that if the initial condition is very close to a disc in $B_{R\sqrt T}(x_0)$ (condition a) and b)) and  $\partial M_t$ lies outside $B_{R\sqrt T}(x_0)$ for all $0\leq t\leq T$ (condition c)), then we get good control of $M_t$ inside $B_{S\sqrt T}(x_0).$
\end{rmk}
\begin{proof}
It suffices to prove this for $T=1$ and $x_0=0$.  Consider a sequence of flows  $(M^i_t)_{0\leq t\leq 1}$ satisfying all the hypothesis with $\delta_i$ converging to zero and $R_i$ tending to infinity.  The sequence of flows $(M^i_t)_{t\geq 0}$ will converge weakly to $(\bar M_t)_{t\geq 0}$, a weak solution to mean curvature flow (see \cite[Section 7.1]{ilmanen1}). The fact that the $C_{loc}^{2,\alpha}$ norm of $M^i_0$ converges to zero implies that $M^i_0$ converges in $C^{2,\alpha}_{loc}$ to a union of planes. From b) we conclude that $M^i_0$ converges to a multiplicity one plane $P$. Because $\partial M^i_t$ lies
 outside $B_{R_i}$ for all $0\leq t\leq 1$ with $R_i$ tending to infinity and 
 $$\lim_{i\to\infty}\int_{M^i_0}\Phi(y,l)d\H^2=\int_{P}\Phi(y,l)d\H^2=1\mbox{ for every $y$ and $l$},$$
we can still conclude from Huisken's monotonicity formula  that for all $i$ sufficiently large
$$ \int_{M^i_t}\Phi(y,l)d\H^2\leq 1+\varepsilon_0\quad\mbox{for all } y\in B_{S+1}, \mbox{ $t\leq 1$, and }l\leq 2.$$
This proves i). Moreover, the above inequality also implies,  via White's Regularity Theorem,  that $M^i_t$ converges in $C^{2,\alpha}_{loc}$ to $P$ for all $\kappa\leq t\leq 1$ and so ii) will also hold for all $i$ sufficiently large. This implies the desired result.
\end{proof}

\subsection{Proof of Theorem \ref{connected} and Theorem \ref{annulus.estimates}}

\begin{proof}[Proof of Theorem \ref{connected}]

We first prove part ii). Consider $\delta$ and $R$ given by Lemma \ref{ap2} when $\kappa=1/2,$ $\nu$ is the constant fixed in Theorem \ref{connected},  and $S$ is large to be chosen later. The same reasoning used in  Remark \ref{rmk.connected} 2) shows the existence of $K_1$ (depending on $R$ and $\delta$) so that  for all $\epsilon$ small and $\underline R$ sufficiently large, the $C^{2,\alpha}$ norm of  $L\cap A(K_1, 2\overline R-K_1)$ is smaller than $\delta/2$, and the area ratios  with scale smaller than $ 2 R$ are close to one. Thus, after relabelling $K_1$ to be $K_1-\sqrt 2 R$ we can apply Lemma \ref{ap2} ii) (with $T=4$) to  $M_0=L$ for all $x_0$ in $\Omega=L\cap A(K_1, 2\overline R-K_1)$ and conclude Theorem \ref{connected} ii). Moreover, we also conclude from Lemma \ref{ap2} i) that
\begin{equation}\label{ober}
\int_{L_t}\Phi(y,l)d\H^2\leq 1+\varepsilon_0\quad\mbox{for all } y\in \Omega(S), \mbox{ $t\leq 4$, and }l\leq 4,
\end{equation}
where $\Omega(S)$ denotes the tubular neighbourhood of $\Omega$ in $\R^n$ with radius $S$.

We now prove part i). From \eqref{ober}, we  see that  hypothesis a) and b) of Lemma \ref{ap1} are satisfied with $T=2$,  $s=0,$ and $r=S-2^{3/2}\Lambda$ (which we assume to be positive). Hence Lemma \ref{ap1} i) gives that the $C^{2,\alpha}$ norm of $L_t$ in $\Omega(S-2^{1/2}\Lambda)$ is bounded by $\Lambda/\sqrt t$. Theorem \ref{connected} i) follows from this   provided 
\begin{equation}\label{inclusion2.connected}
L_t\cap A(K_1, 2\overline R -K_1)\subset \Omega(S-2^{3/2}\Lambda).
\end{equation}
This inclusion follows  because, according to Brakke's Clearing Out Lemma \cite[Section 12.2]{ilmanen1} (which can be easily extended to our setting assuming  small $C^{0,\alpha}$ norm of $|E|$), there is a universal constant $S_0$ such that 
$$L_t\cap A(K_1, 2\overline R-K_1)\subset \Omega(S_0)\mbox{ for all }0\leq t\leq 2.$$
Thus we simply need to require  $S-2^{3/2}\Lambda>S_0$ in order to obtain \eqref{inclusion2.connected}.
Furthermore, Lemma \ref{ap1} ii) implies 
\begin{equation*}
F_t(x)\in \Omega(S-2^{3/2}\Lambda) \implies |F_s(x)-x|< \Lambda s^{1/2}\mbox{ for all } 0\leq s\leq 2
\end{equation*}
which combined with \eqref{inclusion2.connected} gives
\begin{equation}\label{inclusion.connected}
F_t(x)\in A(K_1, 2\overline R-K_1) \implies |F_s(x)-x|<\Lambda s^{1/2}\mbox{ for all } 0\leq s\leq 2
\end{equation}
and this proves the second statement of Theorem \ref{connected} i). 

We now prove the first statement of Theorem \ref{connected} iii). Suppose 
$$L_{t'}\cap B_{\underline R-\sqrt 2\Lambda}\nsubseteq F_{t'}(L\cap B_{\underline R})=Q_{1,t'}\cup Q_{2,t'},$$
meaning $F_{t'}(x)\in B_{\underline R-\sqrt 2\Lambda}$ but $x \notin B_{\underline R}$. By continuity there is $0\leq t\leq t'$ so that 
$$F_t(x)\in A(K_1, \underline R)$$
and this implies from \eqref{inclusion2.connected} that $F_t(x)\in \Omega(S- 2^{3/2}\Lambda)$, in which case we conclude from \eqref{inclusion.connected} that $|F_{t'}(x)-x|<\sqrt 2\Lambda$, a contradiction. Similar reasoning shows the other inclusion in Theorem \ref{connected} iii).

Finally we show iv). Apply Lemma \ref{ap2} with   $S=K_1/\sqrt 2$, $\kappa=1/2$, and $\nu$  the constant fixed in this theorem, to  $M_0=Q_{2,0}=P_3\cap B_{\underline R}$ where $x_0=0$. Note that hypothesis a) and b) of Lemma \ref{ap2} are satisfied with $T=2$ if one assumes $\underline R$ sufficiently large. Moreover, hypothesis c) is also satisfied because due to Theorem \ref{connected} i) we have  $\partial Q_{2,t}\subset A(\underline R-2\Lambda_0, \underline R+2\Lambda_0)$. Thus $Q_{2,t}$ is $\nu$-close in $C^{2,\alpha}(B_{K_1})$ to $P_3$ for every $1\leq t\leq 2$. 
\end{proof}

 \begin{proof}[Proof of Theorem \ref{annulus.estimates}]
 During this proof we will use Theorem \ref{connected} i) and iii) with $\nu=1$. $\Lambda_0$  is the constant given by that theorem.

 From the maximum principle applied to $\theta_t$ we know that
$$\sup_{0\leq t\leq 2}\sup_{Q_{1,t}}|\theta_t|\leq \max\left\{\sup_{Q_{1,0}}|\theta|,\sup_{0\leq t\leq 2}\sup_{\partial Q_{1,t}}|\theta_t|\right\}.$$
The goal now is to control  the $C^1$ norm of $Q_{1,t}$ along $\partial Q_{1,t}$ so that we control $\sup_{\partial Q_{1,t}}|\theta_t|$.

Given $\eta$ small, consider $R$ and $\delta$ given by Lemma \ref{ap2} when $\nu=\eta$, $S=2\Lambda_0$, and $\kappa=1/2$.  We have $$\partial Q_{1,0}=Q^1\cap\{|x|=\underline R\},$$ where $Q^1$ is defined in \eqref{defi.geral}. Thus, for all $\underline R$ sufficiently large and $\varepsilon$ small,  we have that $M_0=Q^1$  satisfies hypothesis a) and b) of Lemma \ref{ap2} for every $x_0\in \partial Q_{1,0}$. Moreover $\partial( F_t(Q^1))\cap B_{\overline R}=\emptyset$   by Theorem \ref{connected} i), and so hypothesis c) is also satisfied  because we are assuming $\overline R\geq 4\underline R$ (see \eqref{defi2.geral}).

This means $F_t(Q^1)\cap B_{2\Lambda_0\sqrt t}(x_0)$ is graphical over $T_{x_0}Q^1$ with  $C^1$ norm being smaller than $\eta$ for all $0\leq t\leq 2$. Thus we can choose $\eta$ small so that 
\begin{equation}\label{angle.geral} 
\sup\{|\theta_t(y)-\theta_0(x_0)| :  y\in F_t(Q^1)\cap B_{2\Lambda_0\sqrt t}(x_0)\}\leq 1/(2K_0).
\end{equation}
Using Theorem \ref{connected} i) we see that  for every $y\in \partial Q_{1,t}$ there is $x_0\in \partial Q_{1,0}$ so that $y\in B_{2\Lambda_0\sqrt t}(x_0)$. Thus we obtain from \eqref{angle.geral}
$$
\sup_{\partial Q_{1,t}}|\theta_t|\leq \sup_{\partial Q_{1,0}}|\theta|+1/(2K_0)
$$
 and  this implies i) because we are assuming $\sup_{Q^1} |\theta|\leq \pi/2-K_0^{-1}$.

 We now prove ii). Assume for a moment that the metric $g_R$ (defined in Section \ref{prelim}) is Euclidean in $B_{2\overline R}$. Because $Q_{1,t}$ is almost-calibrated we have from \cite[Lemma 7.1]{neves} the existence of a constant $C$ depending only  $K_0$  so that,  for every open set   $B$ in $Q_{1,t}$ with rectifiable boundary,
\begin{equation*}
	\left(\H^2(B)\right)^{1/2}\leq C\,\mbox{length}\,(\partial B).
\end{equation*}
It is easy  to recognize  the same is true (for some slightly larger $C$) if $g_R$ is very close to the Euclidean metric.
Set
$$\psi(r)=\H^2\left(\hat{B}_{r}(x)\right)$$ which has, for almost all $r<\dist(y,\partial Q_{1,t})$, derivative given by
$$\psi^{\prime}(r)=\mbox{length}\,\left(\partial \hat{B}_{r}(x)\right)\geq C^{-1}(\psi(r))^{1/2}.$$
Hence, integration implies that for some other constant $C$, 
$\psi(r)\geq C r^2$ and so ii) is proven.

We now prove iii). The Lie derivative of $\lambda_t=F^*_t (\lambda)$ is given by 
$$\L_H \lambda_t=dF_t^*(H\lrcorner\lambda)+F_t^*(H\lrcorner 2\omega)=d(F_t^*(H\lrcorner\lambda)-2\theta_t)$$ 
and so we can find $\beta_t\in C^{\infty}(Q_{1,t})$ with $d\beta_t=\lambda$ and 
\begin{equation}\label{derivative.connected}
\frac{d\beta_t}{dt}=H\lrcorner\lambda-2\theta_t.
\end{equation}
A simple computation shows that $\Delta \beta_t=H\lrcorner\lambda +\sum_{i=1}^2\nabla _{e_i}\lambda (e_i)$ and this proves iii).

We now prove iv). Combining  Theorem \ref{connected}  i) and \eqref{derivative.connected} we have that 
$$\left|\frac{d\beta_t}{dt}(F_t(x))\right|\leq \frac{\Lambda_0}{\sqrt t}|F_t(x)|+\pi-K_0^{-1}$$
for every $ x \in L\cap A(R_1+2\Lambda_0, \underline R)$. Thus after integration in the $t$ variable, assuming $t\leq 2$, and recalling  \eqref{beta.connected}, we obtain a constant $C=C(K_0, \Lambda_0)$ such that
$$|\beta_t(F_t(x))|\leq C(|F_t(x)|+|\beta(x))|+C\leq C(|F_t(x)|^2+1).$$
We are left to estimate $\beta_t$ on $A_t=F_t(Q_{1,0}\cap B_{R_1+2\Lambda_0})$. From Theorem \ref{connected} i) we know that $A_t\subseteq B_{C_1}(0)$ for some $C_1=C_1(K_0, \Lambda_0, R_1)$ and thus, provided we assume  $g_R$ to be sufficiently close to the Euclidean metric,
$$|\nabla \beta_t (x)|=|\lambda|\leq 2C_1 \quad\mbox{for every }x\in A_t.$$
Hence, if we fix $x_1$ in  $\partial A_t$, we can find $C=C(K_0,\Lambda_0, R_1)$ so that  for  every $y$ in $A_t$
$$
|\beta_t(y)| \leq  |\beta_t(x_1)|+C\dist_{A_t}(x_1,y)\leq C(1+\dist_{A_t}(x_1,y)),
$$
where $\dist_{A_t}$ denotes the intrinsic distance in $A_t$. Property ii) of this theorem,  $A_t\subseteq B_{C_1}(0)$, and Lemma \ref{area.bounds}  are enough to bound uniformly the intrinsic diameter of $A_t$  and thus bound   $\beta_t$ uniformly on $A_t$. Hence iv) is proven.

We now prove v). In what follows $E^j_2$ denotes any term with decay $(|x|^j+1)O(1/\overline R)$. Given a coordinate function $v=x_i$ or $y_i$, $i=1,2$, we have
$$\frac{d v}{dt}=\Delta v-\sum_{i=1}^2 g(\nabla_{e_i} V,e_i)=\Delta v+E_2^0,$$
where $V$ denotes the gradient of $v$ with respect to $g_R$.
Thus
$$\frac{d \mu}{dt}=\Delta \mu+E_2^1-2g_R(X_1^{\top},Y^{\top}_2)+2g_R(Y^{\top}_1,X^{\top}_2),$$
where $X_i, Y_i, i=1,2$ denote the gradient of the coordinate functions with respect to $g_R$. If the ambient Calabi-Yau structure were Euclidean, then
\begin{multline*}
\langle X_1^{\top},Y_2^{\top} \rangle-\langle Y_1^{\top},X_2^{\top} \rangle=-\langle (J Y_1)^{\top},Y_2 \rangle-\langle Y_1^{\top},X_2 \rangle\\
=-\langle J Y_1^{\bot},Y_2 \rangle-\langle Y_1^{\top}, X_2 \rangle=-\langle Y_1^{\bot}+Y_1^{\top},X_2 \rangle=-\langle Y_1,X_2\rangle=0.
\end{multline*}
In general, it is easy to see that $g_R(X_1^{\top},Y^{\top}_2)-g_R(Y^{\top}_1,X^{\top}_2)=E_2^0$ and so 
$$\frac{d\mu^2}{dt} \leq \Delta \mu^2-2|\nabla \mu|^2+E_2^3.$$
\end{proof}

\section{Second Step: Self-expanders}\label{self-expanders}
	 The goal of this section is to prove the theorem below.  For the reader's convenience, we recall that the planes $P_1,P_2$ are defined in Section \ref{prelim.n},  $K_0$ is defined at the beginning of Section \ref{geral}, $Q^1$ is defined in  \eqref{defi.geral}, and $Q_{1,t}$ is defined in Theorem \ref{connected} iii). The self-expander equation is defined in Section \ref{prelim.s} and the self-expander $\mathcal{S}$ is defined in \eqref{Q} (see Figure \ref{modelo-se}).
	 
\begin{thm}\label{proximity.final}
Fix $S_0$ and $\nu$. There are $\varepsilon_3$ and $R_3,$ depending on $S_0,$ $\nu$, and $K_0,$ such that if  $\underline R\geq R_3,$ $\varepsilon\leq \varepsilon_3$ in \eqref{defi2.geral}, and
 $$\mbox{the flow }Q_{1,t}\mbox{ exists smoothly for all }0\leq t\leq 2,$$
  then $ t^{-1/2} Q_{1,t}$ is $\nu$-close in $C^{2,\alpha}(B_{S_0})$ to $\mathcal{S}$ for every $1\leq t\leq 2$.   
\end{thm}

As we will see shortly, this theorem  follows from Theorem \ref{proximity.total} below. Recall that, as seen in Theorem \ref{annulus.estimates} iii), we can find $\beta_t$ on $Q_{1,t}$  so that
$$d\beta_t=\lambda=\sum_{i=1}^2 x_idy_i-y_idx_i.$$

 \begin{thm}\label{proximity.total} Fix $S_0$ and $\nu$. 
 
 There are $\varepsilon_4$, $R_4,$ and $\delta$ depending on $S_0,$ $\nu$, and $K_0,$ such that if  $\underline R\geq R_4,$ $\varepsilon\leq \varepsilon_4$ in \eqref{defi2.geral}, and
 \begin{itemize}
 \item $\mbox{the flow }Q_{1,t}\mbox{ exists smoothly for all }0\leq t\leq 2;$
 \item \begin{equation}\label{beta.proximity.total}
\int_{Q^1\cap B_{\underline R}}\beta^2\exp(-|x|^2/8)d\H^2\leq \delta;
\end{equation}
 \end{itemize}
 then $ t^{-1/2} Q_{1,t}$ is $\nu$-close in $C^{2,\alpha}(B_{S_0})$ to a smooth embedded self-expander asymptotic to $P_1$ and $P_2$ for every $1\leq t\leq 2$.   
\end{thm}
\begin{rmk}\label{pro.total.rmk}
 {\bf (1)}  If the ambient metric $g_R$ (defined in Section \ref{prelim}) were Euclidean, then $|\nabla \beta(x)|=|x^{\bot}|$ and thus $\beta$ would be  constant exactly on cones. Hence, roughly speaking,  the left-hand side of \eqref{beta.proximity.total} measures how close $Q^1\cap B_{\underline R}$ is  to a  cone. 
  
{\bf (2)}  The content of the theorem is that given $\nu$ and $S_0$, there is $\delta$ so that  if the initial condition is  $\delta$-close, in the sense of \eqref{beta.proximity.total}, to a non area-minimizing configuration of two planes $P_1+P_2$ and the flow exists smoothly for all $0\leq t\leq 2$, then the flow will be $\nu$-close to a smooth self-expander in $B_{S_0}$ for all $1\leq t\leq 2$.

{\bf (3)}  The result is false if one removes  the hypothesis that the flow exists smoothly for all $0\leq t\leq 2$. For instance, there are known examples \cite[Theorem 4.1]{neves} where $Q_{1,0}$ is very close to $P_1+P_2$ (see \cite[Figure 1]{neves}) and a finite-time singularity happens for a very short time $T$.  In this case $Q_{1,T}$ can be seen as a transverse intersection of small perturbations of $P_1$ and $P_2$ (see \cite[Figure 2]{neves}) and we could continue the flow past the singularity by flowing each component of $Q_{1,T}$ separately, in which case $Q_{1,1}$ would be very close to $P_1+P_2$ and this is not a {\em smooth} self-expander.  The fact the flow exists smoothly will be crucial  to prove Lemma \ref{stationary}.

{\bf (4)}  The result is also false if $P_1+P_2$ is area-minimizing. The reason is that in this case the self-expander asymptotic to $P_1+ P_2$ is simply $P_1+ P_2$, which is singular at the origin and thus not {\em smooth} as it is guaranteed by Theorem \ref{proximity.total}. The fact that $P_1+ P_2$ is  not area-minimizing will be crucial to prove Lemma \ref{stationary}.
  
{\bf (5)} The strategy to prove Theorem \ref{proximity.total} is the following. The first step (Proposition \ref{monotone}) is to show that if the left-hand side of \eqref{beta.proximity.total} is very small, then
$$\int_{Q_{1,1}\cap B_{\underline R/2}}(\beta_1+2\theta_1)^2\Phi(0,4-t)d\H^2+\int_0^2 \int_{Q_{1,t}}|x^{\bot}-2tH|^2\Phi(0,4-t)d\H^2 dt$$
is also very small. The second step (Proposition \ref{proximity_se}) in the proof will be to show  that if
$$\int_{Q_{1,1}\cap B_{\underline R/2}}(\beta_1+2\theta_1)^2\Phi(0,4-t)d\H^2+\int_0^2 \int_{Q_{1,t}}|x^{\bot}-2tH|^2\Phi(0,4-t)d\H^2 dt$$
is very small, then $t^{-1/2}Q_{1,t}$ will be $\nu$-close in $C^{2,\alpha}(B_{S_0})$ to a smooth self-expander. It is in this step that we use the fact that the flow exists smoothly and $P_1+ P_2$ is not an area-minimizing configuration.
\end{rmk}

\begin{proof}[Proof of Theorem \ref{proximity.final}] The first step is to show that Theorem \ref{proximity.total} can be applied, which amounts to show that \eqref{beta.proximity.total} holds if we choose $\varepsilon$ sufficiently small and $\underline R$ sufficiently large. Thus, we obtain that $ t^{-1/2} Q_{1,t}$ is $\nu$-close in $C^{2,\alpha}(B_{S_0})$ to a smooth embedded self-expander asymptotic to $P_1$ and $P_2$ for every $1\leq t\leq 2$. The second step is to show that  self-expander must be $\mathcal{S}$.
\vskip 0.03in
\noindent{\bf First Step:} We note  $Q^1\cap A(1, \underline R)$ (defined in \eqref{defi.geral}) coincides with $(P_1\cup P_2)\cap  A(1, \underline R)$ and so the uniform control we have on $\beta$ given by \eqref{beta.connected} implies that for all $\delta$ there is $r_1$ large depending on $K_0$ and $\delta$ so that
\begin{equation}\label{llll}
\int_{Q^1 \cap A(r_1, \underline R)}\beta^2\exp(-|x|^2/8)d\H^2\leq \frac{\delta}{2}
\end{equation}
for all $\varepsilon$ small and $\underline R$ large. Also, if we make $\varepsilon$ tend to zero and $\underline R$ tend to infinity in \eqref{defi2.geral}, it is straightforward to see that $Q^1$  tends to $P_1\cup P_2$ smoothly on any compact set which does not contain the origin. Because $\beta$ is constant on cones, we can choose $\beta$ on $Q^1$ so that 
$$
 \lim_{\varepsilon\to 0, \underline R\to \infty} \int_{Q^1\cap B_{r_1}}\beta^2\exp(-|x|^2/8)d\H^2=0.
$$
Combining this with \eqref{llll} we obtain that for all $\varepsilon$ small and $\underline R$ large
$$\int_{Q^1\cap B_{\underline R}}\beta^2\exp(-|x|^2/8)d\H^2\leq \delta.$$
Hence all the hypothesis of Theorem \ref{proximity.total} hold
\vskip 0.03in
\noindent{\bf Second Step:} Let $Q$ denote a smooth embedded Lagrangian self-expander  asymptotic to $P_1+P_2$.
Then $Q_t=\sqrt t Q$ and $\lim_{t\to0^+} Q_t=P_1+P_2$  as Radon measures. Thus, if we recall the function $\mu=x_1y_2-y_1x_2$ defined in Theorem \ref{annulus.estimates}  v), we have 
 \begin{equation}\label{animals}
 \lim_{t\to 0^+}\int_{Q_t}\mu^2\Phi(0,T-t)d\H^2=\int_{P_1+P_2}\mu^2\Phi(0,T)d\H^2=0.
 \end{equation}
 Using the evolution equation  for $\mu$ given in Theorem \ref{annulus.estimates} v) ($E_2$ is identically zero)  into Huisken's Monotonicity Formula (see Lemma \ref{huisken1}) we have
$$\frac{d}{dt}\int_{Q_t}\mu^2\Phi(0,T-t)d\H^2\leq 0.$$
This inequality and \eqref{animals} imply at once that
 $$\int_{Q_t}\mu^2\Phi(0,1)d\H^2=0\mbox{ for all $t\geq 0$}$$
 and so $Q\subset \mu^{-1}(0)$.
 A trivial modification of Lemma \ref{ap3} implies the existence of $\gamma$ asymptotic to $\chi$ (the curve defined in \eqref{Q}) so that
 $$Q=\{(\gamma(s)\cos \alpha, \gamma(s)\sin \alpha)\,|\, s\in \R,\alpha \in S^1\}.$$
 From \cite[Section 5]{anciaux} we know that $\chi=\gamma$ and so the result follows.
\end{proof}

\subsection{Proof of Theorem \ref{proximity.total}}
Throughout this proof we assume that $\underline R$ is sufficiently large and $\varepsilon$ is sufficiently small so that Theorem \ref{connected}  (with $\nu=1$) and Theorem \ref{annulus.estimates} apply. We also assume the flow $(Q_{1,t})_{0\leq t\leq 2}$ exists smoothly.

 For simplicity, denote $Q_{1,t}$ simply by $Q_t$.  We also recall that the constant $K_0$ which will appear multiple times during this proof was defined at the beginning of Section \ref{geral}.

\begin{prop}\label{monotone} Fix $\eta$. 
 
 There are  $\varepsilon_5$ and  $R_5$ depending on $\eta$ and $K_0$ so that if  $\varepsilon\leq \varepsilon_5$ and  $\underline R\geq R_5$  in \eqref{defi2.geral}, then 
\begin{multline*}
	\sup_{0\leq t\leq 2}\int_{Q_{t}\cap B_{\underline R/2}}(\beta_t+2t\theta_t)^2\Phi(0,4-t)d\H^2\\
	+\int_0^2 \int_{Q_{t}\cap B_{\underline R/2}}|x^{\bot}-2tH|^2\Phi(0,4-t)d\H^2 dt\\
	\leq \frac{\eta}{2}+\int_{Q^1\cap B_{\underline R}}\beta^2\Phi(0,4)d\H^2.
\end{multline*}
\end{prop}
\begin{rmk}\label{monotone.remark}
The idea is to apply  Huisken monotonicity formula fot $(\beta_2+2t\theta_t)^2$. Some extra (technical) work has to be done because $Q_{t}$ has boundary and the ambient metric $g_R$ (defined in Section \ref{geral}) is not Euclidean.
\end{rmk}
\begin{proof}
Let $\phi \in C^{\infty}(\R^{4}) $ such that $0\leq \phi\leq 1,$ 
$$\phi=1\mbox{  on }B_{\underline R/2},\quad\phi=0\mbox{ on }B_{2\underline R/3}, \quad |D\phi|+|D^2\phi| \leq  \frac{\Lambda}{\underline R},$$
where $\Lambda$ is some universal constant. By Theorem \ref{connected}  i) we have that,  provided we chose $\underline R$  large and $\varepsilon$  small, $\partial Q_t \cap B_{\underline R/2}=\empty$ and thus $\phi$ has compact support in $Q_t$.

Set
$\gamma_t=\beta_t+2t\theta_t$
and so on $Q_t$ we have from Theorem \ref{annulus.estimates} iii)
$$
	\frac{d(\gamma_t\phi)^2}{dt}=\Delta(\gamma_t \phi)^2-2|\nabla \gamma_t|^2\phi^2+\langle H, D\phi^2\rangle \gamma_t^2\\
	-2  \langle \nabla \gamma_t^2, D\phi^2 \rangle -\gamma_t^2\Delta \phi^2+\phi^2E_1
$$	
Thus, using Theorem \ref{connected} i)  to estimate $H$ and Theorem 	\ref{annulus.estimates} iv), we have that for all $\underline R$ large and $\varepsilon$ small
$$
  \frac{d(\gamma_t\phi)^2}{dt}\leq \Delta(\gamma_t \phi)^2-2|\nabla \gamma_t|^2\phi^2+ \frac{C_1}{\sqrt t \underline R}(|x|^4+1)(1-\chi_{\underline R/2})+|E_1|,
$$	
where $C_1=C_1(K_0, \Lambda_0, D_1)$ and $\chi_{\underline R/2}$ denotes the characteristic function of $B_{\underline R/2}$.  From  Lemma \ref{huisken1} we conclude
\begin{multline*}
	\frac{d}{dt}\int_{Q_t}(\gamma_t \phi)^2\Phi(0,4-t)d\H^2 +2\int_{Q_t}|\nabla \gamma_t|^2\phi^2\Phi(0,4-t)d\H^2\\
	\leq \int_{Q_t}\left(\gamma_t^2\frac{|E|^2}{4}+|E_1|\right)\Phi(0,4-t)d\H^2+ \frac{C_1}{\sqrt t \underline R}\int_{Q_t\setminus B_{\underline R/2}}(|x|^4+1)\Phi(0,4-t)d\H^2
\end{multline*}

We now estimate the two terms on the right-hand side. If $g_R$ (defined at the beginning of Section \ref{geral}) were Euclidean,  both terms $|E|^2$ and $E_1$ mentioned above would vanish. Otherwise it is easy to see that making $\underline R$ sufficiently large so that $g_R$ becomes close to Euclidean, both terms can be made arbitrarily small. The growth of $\gamma_t$ is quadratic (Theorem 	\ref{annulus.estimates} i) and iv)) and so choosing $\underline R$ sufficiently large and $\varepsilon$ sufficiently small we have
$$\int_{Q_t}\left(\gamma_t^2\frac{|E|^2}{4}+|E_1|\right)\Phi(0,4-t)d\H^2\leq \frac \eta 8\quad\mbox{for all }t\leq 2.$$
Using that $|x|\geq  |x|^2/2+\underline R^2/8$ outside $B_{\underline R/2}$, it is easy to see that
$$\Phi(0,4-t)\leq 2^{1/2}\Phi(0,2(4-t))\exp(-\underline R^2/(32(4-t))) \mbox{ on }\R^4\setminus B_{\underline R/2}.$$
Thus, for all $0\leq t\leq 2$, the uniform area bounds given in Lemma \ref{area.bounds} imply
\begin{multline*}
\int_{Q_t\setminus B_{\underline R/2}}(|x|^4+1)\Phi(0,4-t)d\H^2\\
\leq  C_2\exp(-\underline R^2/C_2)\int_{Q_t\setminus B_{\underline R/2}}(|x|^4+1)\Phi(0,2(4-t))d\H^2\leq C_3\exp(-\underline R^2/C_3),
\end{multline*}
where $C_2$ and $C_3$ depend only on $K_0$. Therefore we have
\begin{multline*}
	\frac{d}{dt}\int_{Q_t}(\gamma_t \phi)^2\Phi(0,4-t)d\H^2 +2\int_{Q_t}|\nabla \gamma_t|^2\phi^2\Phi(0,4-t)d\H^2\\
	\leq \frac{C_4}{\sqrt t \underline R}\exp(-\underline R^2/C_4)+\frac \eta 8,
\end{multline*}
where $C_4=C_4(C_1, C_3)$.
Integrating this inequality we obtain for all $t\leq 2$
\begin{multline}\label{inequality.monotone}
	\int_{Q_t}\gamma_t^2\phi^2\Phi(0,4-t)d\H^2 +2\int_0^t\int_{Q_s}|\nabla \gamma_t|^2\phi^2\Phi(0,4-s)d\H^2 ds\\
	\leq \int_{Q^1\cap B_{\underline R}}\beta^2\Phi(0,4)d\H^2+2^{3/2}C_4\underline R^{-1}\exp(-\underline R^2/C_4)+\frac{\eta}{4}.
\end{multline}
If the metric $g_R$ were Euclidean then $|\nabla\gamma_t|^2=|x^{\bot}-2tH|^2$. Hence the result follows from \eqref{inequality.monotone} if we assume $\underline R$ is large enough  so that $2|\nabla\gamma_t|^2\leq |x^{\bot}-2tH|^2$
and $$2^{3/2}C_4\underline R^{-1}\exp(-\underline R^2/C_4)\leq \frac{\eta}{4}.$$
\end{proof}

The next proposition is crucial to prove Theorem \ref{proximity.total}.

\begin{prop}\label{proximity_se}
Fix $\nu$ and $S_0$.

 There are $\varepsilon_6,$  $R_6,$ and $\eta$ depending on $\nu$, $K_0,$ and $S_0$,  such that if  $\underline R\geq R_5,$ $\varepsilon\leq \varepsilon_5$ in \eqref{defi2.geral}, and
\begin{multline}\label{rmk.se.equation}
	\sup_{0\leq t\leq 2}\int_{Q_{t}\cap B_{\underline R/2}}(\beta_t+2t\theta_t)^2\Phi(0,4-t)d\H^2\\
	+\int_0^2 \int_{Q_{t}\cap B_{\underline R/2}}|x^{\bot}-2tH|^2\Phi(0,4-t)d\H^2 dt \leq {\eta}.
\end{multline}
 then $ t^{-1/2} Q_t$ is $\nu$-close in $C^{2,\alpha}(B_{S_0})$ to a smooth embedded self-expander asymptotic to $P_1+P_2$ for every $1\leq t\leq 2$.
\end{prop}
\begin{rmk}\label{proximity_se.rmk}
The strategy to prove this proposition is the following. We argue by contradiction and first principles will give us a sequence of flows $(Q^i_t)_{0\leq t\leq 2}$ converging weakly to a Brakke flow $(\bar Q_t)_{0\leq t\leq 2}$, where  in \eqref{defi2.geral} we have $\underline R_i$ tending to infinity, $\varepsilon_i$ tending to zero, and
\begin{multline}\label{eqeq.se}
\lim_{i\to\infty} \int_{Q^i_{1}\cap B_{\underline R_i/2}}(\beta^i_1+2\theta_1^i)^2\Phi(0,4-t)d\H^2 \\+ \int_0^2 \int_{Q^i_{t}\cap B_{\underline R_i/2}}|x^{\bot}-2tH|^2\Phi(0,4-t)d\H^2 dt =0.
\end{multline}
Standard arguments (Lemma \ref{time.zero}) imply $\bar Q_t$ is a self-expander with $$\lim_{t\to 0^+} \bar Q_t=P_1+ P_2.$$
The goal is to show that $\bar Q_1$ is smooth because we could have, for instance,  $\bar Q_1=P_1+P_2$. 

The first step (Lemma \ref{stationary}) is to show that $\bar Q_1$ is not stationary  and the idea is the following. If $\bar Q_1$ were stationary then $\bar Q_t=\bar Q_1$ for all $t$ and so $\bar Q_1=\lim_{t\to 0^+} \bar Q_t=P_1+P_2.$ On the other hand,  from the control given in Theorem \ref{connected}, we will be able to find $r_1>0$ large so that $Q^i_1\cap B_{r_1}$ is connected (if the flow had a singularity this would not necessarily be true). Furthermore, we will deduce  from \eqref{eqeq.se} that
$$
\int_{Q^i_1\cap B_{r_1}}|\nabla\beta^i_1|^2d\H^2=\lim_{i\to\infty}  \int_{Q^i_1\cap B_{r_1}}|x^{\bot}|^2d\H^2=0.
$$
Hence we can invoke \cite[Proposition A.1]{neves} and conclude that $\beta^i_1$ must tend to constant $\bar \beta$ in $L^2$. Combining this with  \eqref{eqeq.se} we have
$$\lim_{i\to\infty}\int_{Q^i_1\cap B_{r_1}} (\bar \beta+2\theta_1^i)^2d\H^2=0,$$
and thus  $\bar Q_1$ must be Special Lagrangian with Lagrangian angle $- \bar \beta/2$. This contradicts the choice of $P_1$ and $P_2$.

The second step (Lemma \ref{small}) is to show the existence of $l_1$ so that,  for every $y\in \C^2$ and $l<l_1$, the Gaussian density ratios of $\bar Q_1$ centered at $y$ with scale $l$ defined by  
$$\Theta(y,l)=\int_{\bar Q_1}\Phi(y,l)d\H^2$$
are very close to one. If true then standard theory implies $\bar Q_1$ is smooth and embedded.  The (rough) idea for the second  step is the following. If  this step fails for some $y\in \C^2$, then $y$ should be in the singular set of  $\bar Q_1$. Now  $T_y\bar Q_1$ should be a union of (at least two) planes. Hence the Gaussian density ratios of $\bar Q_1$ at $y$ for all small scales should  not only be away from one but actually  bigger or equal than two. We know from Huisken's monotonicity formula that the Gaussian  density ratios of $\bar Q_1$ at $y$ and scale $l$ are bounded from above by the Gaussian  density ratios of $\bar Q_0=P_1+P_2$ at $y$ and scale $l+1$. But this latter Gaussian ratios are never bigger than two (see Remark \ref{small.rmk}), which means equality must hold in Huisken's monotonicity formula and so $\bar Q_t$  must be a self-shrinker. Now $\bar Q_t$ is also a self expander and thus it must be stationary. This contradicts the first step.

\end{rmk}
\begin{proof}
	Consider  a sequence  $(\underline R_i)$ converging to infinity and  a sequence $(\varepsilon_i)$ converging to zero in \eqref{defi2.geral} which give rise to a sequence of smooth flows $(Q^i_t)_{0\leq t\leq 2}$ satisfying 
\begin{multline}\label{inequality.proximity_se}
	\sup_{0\leq t\leq 2}\int_{Q^i_{t}\cap B_{\underline R^i/2}}(\beta_t+2t\theta_t)^2\Phi(0,4-t)d\H^2\\
	+\int_0^2 \int_{Q_{t}\cap B_{\underline R_i/2}}|x^{\bot}-2tH|^2\Phi(0,4-t)d\H^2 dt \leq \frac{1}{i}.
\end{multline}
We will show the existence of a smooth self-expander $\bar Q_1$  asymptotic to $P_1$ and $P_2$ so that, after passing to a subsequence, $ t^{-1/2} Q^i_t$ converges in $C^{2,\alpha}(B_{S_0})$ to $\bar Q_1$ for every $1\leq t\leq 2$.

From  compactness for integral Brakke motions \cite[Section 7.1]{ilmanen1} we know that, after passing to a subsequence, $(Q^i_t)_{0\leq t\leq 2}$ converges to an integral Brakke motion $(\bar Q_t)_{0\leq t\leq 2}$, where $Q^i_0$  converges in the varifold sense to the varifold $P_1+P_2$. Furthermore
$$\lim_{i\to\infty}\int_0^2 \int_{Q^i_{t}\cap B_{\underline R_i/2}}|x^{\bot}-2tH|^2\Phi(0,4-t)d\H^2 dt =0,$$
which means 
\begin{equation}\label{se.proximity_se}
H=\frac{x^{\bot}}{2t}\mbox{ on }\bar Q_t\mbox{ for all }t>0
\end{equation}
and so 
$\bar Q_t=\sqrt t \bar Q_1$ as varifolds for every $t>0$ (see proof of \cite[Theorem 3.1]{neves2} for this last fact).

\begin{lemm}\label{time.zero} As $t$ tends to zero, $\bar Q_t$ converges, as Radon measures, to $P_1+P_2$.
\end{lemm}
\begin{rmk}
This lemma is needed because the Brakke flow theory only assures that the support of the Radon measure obtained from $\lim_{t\to 0}\bar Q_t$ is contained in the support of $\lim_{i\to\infty}Q^i_0=P_1+ P_2$.
\end{rmk}
\begin{proof}
 Set
 $$\mu_t(\phi)=\int_{\bar Q_t}\phi \,d\H^2.$$
The Radon measure $\nu=\lim_{t\to 0^+} \mu_t$ is well defined by \cite[Theorem 7.2]{ilmanen1}  and satisfies, for every $\phi\geq 0$ with compact support 
\begin{equation}\label{measure}
\nu(\phi)\leq \lim_{i\to\infty}\int_{Q^i_0}\phi \,d\H^2=\int_{P_1+P_2} \phi \,d\H^2.
\end{equation}
It is simple to recognize that $\nu$ must be either zero, $P_1$, $P_2$, or $P_1+P_2$.

The measure $\nu$ is invariant under scaling meaning that if we set $\phi_{c}(x)=\phi(cx)$ then
\begin{multline*}
	\nu(\phi_c)=\lim_{t\to 0^+}\int_{\bar Q_t}\phi_c d\H^2=c^{-2}\lim_{t\to 0^+}\int_{c\bar Q_{t}}\phi d\H^2=
	c^{-2}\lim_{t\to 0^+}\int_{\bar Q_{c^2t}}\phi d\H^2\\
	=c^{-2}\lim_{t\to 0^+}\int_{\bar Q_{t}} \phi d\H^2=c^{-2}\nu(\phi). 
\end{multline*}
From  Theorem \ref{connected} i) and Theorem \ref{annulus.estimates} ii)  we have that the support of $\nu$ contains $(P_1+P_2)\cap A(K_1,\infty)$  which, combined with the invariance of the measure we just mentioned, implies  the support of $\nu$ coincides with $P_1\cup P_2$.  Thus $\nu=P_1+P_2$ as we wanted to show.

\end{proof}
 
\begin{lemm}\label{stationary}
$\bar Q_1$ is not stationary. 
\end{lemm}
\begin{proof}
	If true, then $\bar Q_1$ needs to be a cone because $x^{\bot}=2H=0$ and  so, because $\bar Q_t=\sqrt t \bar Q_1$, they are also cones for all $t>0$. Hence we must have (from varifold convergence) that for  every $r>0$
			$$\lim_{i\to\infty}\int_{0}^2\int_{Q^i_t\cap B_r}|x^{\bot}|^2d\H^2 dt=0$$
			which implies from \eqref{inequality.proximity_se} that
			\begin{equation*}
				\lim_{i\to\infty}\int_{0}^2 \int_{Q_t^i\cap B_r} (t^2|H|^2+|x^{\bot}|^2)d\H^2 dt=0.
			\end{equation*}
			Therefore, we can  assume without loss of generality that for every $r>0$
			\begin{equation}\label{zero}
			\lim_{i\to\infty}\int_{Q_1^i\cap B_r} (|H|^2+|x^{\bot}|^2)d\H^2=0
			\end{equation}
			and thus, by  \cite[Proposition 5.1]{neves},  $\bar Q_1$ is a union of Lagrangian planes  with possible multiplicities. We will argue that $\bar Q_1$ must be a Special Lagrangian, i.e., all the planes in $\bar Q_1$ must have the same Lagrangian angle. This gives us a contradiction for the following reason: On one hand, $\bar Q_t=\bar Q_1 \mbox{ for all }t>0$ which means $\lim_{t\to 0}\bar Q_t=\bar Q_1$. On the other hand, from  Lemma \ref{time.zero}, we have $\lim_{t\to 0}\bar Q_t=P_1+P_2$ which means $\bar Q_1=P_1+P_2$ and therefore the Lagrangian angle of $P_1$ and $P_2$ must be the identical (or differ by a multiple of $\pi$). This contradicts how $P_1$ and $P_2$ were chosen.

 	From Theorem \ref{connected} ii) (which we apply with $\nu=1$) we have that for all $i$ sufficiently large, $Q^i_1\cap A (R_1,\underline R_i/2)$ is graphical over $(P_1\cup P_2)\cap A (R_1,\underline R_i/2)$ with  $C^{2,\alpha}$ norm uniformly bounded. Hence we can find $r_1 \geq R_1$ so that  if we set  $N_i= Q^i_1\cap B_{3r_1}$ we have for all $i$ sufficiently large that $N_i\cap B_{2r_1}$ connected.  We note that if $Q^i_t$ had a singularity for some $t<1$ then $N_i$ could be two discs intersecting transversally near the origin and thus $N_i\cap B_{2r_1}$ would not be connected. 
	
	Furthermore we obtain from  \eqref{zero} that
	$$ \lim_{i\to\infty}\int_{N_i}|\nabla \beta^i|^2d\H^2 =\lim_{i\to\infty}\int_{N_i}|x^{\bot}|^2d\H^2=0,$$
	and so, because of Theorem \ref{annulus.estimates} ii),  we can apply \cite[Proposition A.1]{neves} and conclude the existence of a  constant $\bar \beta$ so that, after passing to a subsequence, 
		\begin{equation}\label{poincare.se}
		\lim_{i\to\infty}\int_{N_i\cap B_{r_1}}(\beta^i_1-\bar\beta)^2d\H^2=0.
		\end{equation}
Recall that from \eqref{inequality.proximity_se} we have
$$\lim_{i\to\infty} \int_{Q^i_1\cap B_{r_1}}(\beta^i_1+2\theta_1^i)^2d\H^2=0,$$
which combined with \eqref{poincare.se} implies
$$\lim_{i\to\infty} \int_{Q^i_1\cap B_{r_1}}(\bar\beta+2\theta_1^i)^2d\H^2=0.$$
Therefore $\bar Q_1$ must be a Special Lagrangian cone with Lagrangian angle $-\bar \beta/2$.
\end{proof}

In the next lemma, $\varepsilon_0$ denotes the constant given by White's Regularity Theorem \cite{white}.
\begin{lemm}\label{small}
There is $l(t)$, a positive continuous function of $0<t\leq 2$,  so that
	$$\int_{\bar Q_t}\Phi(y,l)d\H^2\leq 1+\varepsilon_0/2\quad\mbox{ for every } l\leq l(t),\,\,y\in \R^4,\mbox{ and }t>0.$$
\end{lemm}
\begin{rmk}\label{small.rmk}
During the proof the following simple formula will be used constantly. Given $y\in \C^2$ ,let $d_1,d_2$ denote, respectively, the distance from $y$ to $P_1$ and $P_2$. Then
\begin{equation}\label{gau.se}
\int_{P_1+P_2}\Phi(y,l)d\H^2=\exp(-d_1^2/(4l))+ \exp(-d_2^2/(4l))\leq 2.
\end{equation}
\end{rmk}
\begin{proof}
	It suffices to prove the lemma for $t=1$ because, as we have seen, $\bar Q_t=\sqrt t \bar Q_1$ for all $t>0$. 
	\vskip 0.05in
	\noindent{\bf Claim:} There is $C_1$  such that for every $l\leq 2$ and $y\in \R^4$
	\begin{equation}\label{density}
		\int_{\bar Q_1}\Phi(y,l)d\H^2\leq 2-C_1^{-1}.
	\end{equation}
From the monotonicity formula for Brakke flows \cite[Lemma 7]{ilmanen}
\begin{multline}\label{huisken}
	\int_{\bar Q_1}\Phi(y,l)d\H^2+\int_{0}^1\int_{\bar Q_t}\left|H+\frac{(x-y)^{\bot}}{2(l+1-t)}\right|^2\Phi(y,l+1-t)d\H^2 dt\\
	= \int_{P_1+P_2}\Phi(y,l+1)d\H^2\leq 2.
\end{multline}
Suppose there is a sequence $(y_i)$ and $(l_i)$ with $0\leq l_i\leq 2$ such that

$$\int_{\bar Q_1}\Phi(y_i,l_i)d\H^2\geq 2-\frac 1 i.$$
Then, from \eqref{huisken} we obtain 
$$\lim_{i\to\infty}\int_{P_1+P_2}\Phi(y_i,l_i+1)d\H^2=2$$
and so, from \eqref{gau.se}, $(y_i)$ must converge to zero. Assuming  $(l_i)$ converges to $\bar l$, we have again from \eqref{huisken} that
 \begin{multline*}
\int_{0}^{1/2}\int_{\bar Q_t}\left|H+\frac{x^{\bot}}{2(\bar l+1-t)}\right|^2\Phi(0,\bar l+1-t)d\H^2 dt\\
\leq \lim_{i\to\infty}\int_{0}^{1}\int_{\bar Q_t}\left|H+\frac{(x-y_i)^{\bot}}{2(l_i+1-t)}\right|^2\Phi(y_i,l_i+1-t)d\H^2 dt\\
\leq 2-\lim_{i\to\infty}\int_{\bar Q_1}\Phi(y_i,l_i)d\H^2 =0.
\end{multline*}
As a result
$$H+\frac{x^{\bot}}{2(\bar l+1-t)}=0\mbox{ on } \bar Q_t \mbox{ for all } 0\leq t\leq 1/2$$
and combining this with the fact that 
$H=\frac{x^{\bot}}{2t}\quad\mbox{on }\bar Q_t$
we obtain that $H=0$ on $\bar Q_1=t^{-1/2}\bar Q_t$, which contradicts Lemma \ref{stationary}. Thus, \eqref{density} must hold.

	To finish the proof we argue again by contradiction and assume the lemma does not hold. Hence, there is a sequence  $(y_j)_{j\in\N}$ of points in $\R^4$ and a sequence $(l_j)_{j\in\N}$ converging to zero for which
	\begin{equation}\label{bound.se}
	\int_{\bar Q_1} \Phi(y_j, l_j) d\H^2\geq 1+\frac{\varepsilon_0}{2}.
	\end{equation}
	
	The first thing we do is to show \eqref{bound.se} implies the existence of $m$ so that $|y_j|\leq m$ for all $j$. The reason is that from \eqref{huisken} we obtain 
	$$ \int_{P_1+P_2}\Phi(y_j,l_j+1)d\H^2\geq 1+\frac{\varepsilon_0}{2}$$
	and so, because $(l_j)$ tends to zero, we obtain from \eqref{gau.se} that the sequence  $(y_j)$ must be bounded.
	
	The motivation for the rest of the argument is the following. The sequence $(y_j)$ has a subsequence which converges to $\bar y\in \C^2$. From \eqref{bound.se} we have that $\bar y$ must belong to the singular set of $\bar Q_1$. The tangent cone to $\bar Q_1$ at $\bar y$ is a union of (at least two) Lagrangian planes and thus for all $l$ very small we must have
	$$ \int_{\bar Q_1}\Phi(\bar y,l)d\H^2\geq 2-\frac{1}{2C_1}.$$
	This contradicts \eqref{density}.

	Recalling that the flow $(Q^i_t)_{0\leq t\leq 2}$ tends to $(\bar Q_t)_{0\leq t\leq 2}$, a standard diagonalization argument allows us to find a sequence of integers $(k_j)_{j\in \N}$ so that the  blow-up sequence 
	$$\tilde Q^j_s=l_j^{-1/2}\left(Q^{k_j}_{1+sl_j}-y_j\right), \quad 0\leq s \leq 1$$
	has
	\begin{equation}\label{first.se}
	-\frac{1}{j}\leq\int_{\tilde Q^j_0}\Phi(0,u)d\H^2- \int_{ l_j^{-1/2}(\bar Q_1-y_j)} \Phi(0, u) d\H^2\leq \frac{1}{j}
	\end{equation}
	for every $1\leq u\leq j$ and
	\begin{equation}\label{second.se}
	\int_1^{1+l_j}\int_{Q^{k_j}_t\cap B_1(y_j)}\left|H-\frac{x^{\bot}}{2t}\right|^2d\H^2 dt\leq l_j^2.
	\end{equation}
	Thus, for every $r>0,$ we have from \eqref{second.se} and $|y_j|\leq m$ that
	\begin{multline*}
		\int_0^1\int_{\tilde Q^j_s\cap B_r(0)}|H|^2d\H^2 ds= l_j^{-1}\int_1^{1+l_j}\int_{Q^{k_j}_t\cap B_{\sqrt l_j r}(y_j)}|H|^2d\H^2 dt\\
		\leq l_j^{-1}\int_1^{1+l_j}\int_{Q^{k_j}_t\cap B_{\sqrt l_j r}(y_j)}\left|H-\frac{x^{\bot}}{2t}\right|^2+\left|\frac{x^{\bot}}{2t}\right|^2d\H^2 dt\leq l_j+C_2 l_j\\
	\end{multline*}
	where $C_2=C_2(r,m,K_0)$.
	Therefore
	$$\lim_{j\to\infty}\int_0^1\int_{\tilde Q^j_s\cap B_r(0)}|H|^2d\H^2 ds=0$$
	and so $(\tilde Q^j_s)_{0\leq s\leq 1}$ converges to an integral Brakke flow $(\tilde Q_s)_{0\leq s\leq 1}$ with $\tilde Q_s=\tilde Q$ for all $s$. From Proposition 5.1 in \cite{neves} we conclude that $\tilde Q$ is a union of Special Lagrangian currents.
 Note that
	$$\int_{\tilde Q} \Phi(0,1)d\H^2\geq 1+\varepsilon_0$$ and so $\tilde Q$ cannot be a plane with multiplicity one. The blow-down $C$ of $\tilde Q$ is a union of Lagrangian planes (those are the only Special Lagrangian cones in $\R^4$) and so
	\begin{equation}\label{blowdown.se}
	\lim_{u\to\infty}\lim_{j\to\infty}\int_{\tilde Q^j_0}\Phi(0,u)d\H^2=\lim_{u\to\infty}\int_{\tilde Q}\Phi(0,u)d\H^2=\int_{C}\Phi(0,1)d\H^2 \geq 2.
	\end{equation}
	From \eqref{blowdown.se} and \eqref{first.se}  one can find $u_0$  such that for every $j$ sufficiently large we have
	\begin{multline*}
		2-\frac{1}{2C_1}\leq\int_{\tilde Q^j_0}\Phi(0,u_0)d\H^2\leq \int_{l_j^{-1/2}(\bar Q_1-y_j)}\Phi(0,u_0)d\H^2+\frac{1}{j}\\
		=\int_{\bar Q_1}\Phi(y_j,u_0l_j)d\H^2+\frac{1}{j}.	
	\end{multline*}
	This contradicts \eqref{density} for all $j$ large.
\end{proof}

The lemma we have just proven allows us to find $l_0$ so that for all $\hat R$ and all $i$ sufficiently large 
$$\int_{Q^i_t} \Phi(y,l)d\H^2\leq 1+\varepsilon_0\quad\mbox{for all }y\in B_{\hat R},\,\, l\leq l_0, \mbox{ and }\frac 1 2\leq t\leq 2.$$
Thus, we have from White's Regularity Theorem \cite{white} uniform bounds on the second fundamental form and all its derivatives on compact sets of  $Q^i_t$ for all $1\leq t\leq 2.$ This implies $\bar Q_t$ is smooth and  $t^{-1/2}Q^{i}_t$ converges  in $C_{loc}^{2,\alpha}$ to $\bar Q_1$, a smooth self expander asymptotic to $P_1+P_2$ by Lemma \ref{time.zero}, which must be embedded due to Lemma \ref{small}. This finishes the proof of Proposition \ref{proximity_se}.

\end{proof}
 
	 Apply Proposition \ref{proximity_se} with $\nu$ and $S_0$ given by Theorem \ref{proximity.total} and then apply Proposition \ref{monotone} with $\eta$ being the one given by Theorem  \ref{proximity.total}. Theorem  \ref{proximity.total} follows at once if we choose $\delta=\eta/2$, $\varepsilon_3=\min\{\varepsilon_5,\varepsilon_6\},$ and $\underline R_3=\max\{\underline R_5,\underline R_6\}$.

\section{Third Step: Equivariant flow}\label{sta}
\subsection{Setup of Section \ref{sta}}\label{setup.sta}
Consider a smooth curve $\sigma:[0,+\infty)\longrightarrow \C$  so that 
\begin{itemize}
\item $\sigma^{-1}(0)=0$ and $\sigma\cup-\sigma$ is smooth at the origin;
\item $\sigma$ has a unique self intersection;
\item Outside a large ball the curve $\sigma$ can be written as the graph of a function $u$ defined over part of the negative real axis with  $$\lim_{r\to-\infty}|u|_{C^{2,\alpha}((-\infty,r])}=0;$$
\item For some $a$ small enough we have
\begin{equation}\label{equiv.cone}
\sigma\subseteq C_{a}=\{r\exp(i\theta)\,|\, r\geq 0, \pi/2+2a <\theta<\pi+a\}.
\end{equation}
\end{itemize}
The curve $\sigma$ shown in  Figure \ref{modelo-3} has all these properties. Condition \eqref{equiv.cone} is there for technical reasons which will be used during Lemma \ref{equiv.sturm}.

 Denote by $A_1$ the area enclosed by the self-intersection of $\sigma$.

We assume that $L\subset M$ is a Lagrangian surface as defined in \eqref{defi2.geral} and  that $\varepsilon$, $\underline R$ are such that Theorem \ref{connected} (with $\nu=1$) and Theorem \ref{annulus.estimates} hold. We also assume that the solution to Lagrangian mean curvature flow $(L_t)_{t\geq 0}$ satisfies the following condition.
\begin{itemize}
\item[($\star$)]There is a constant $K_1$, a disc  $D$, and $F_t: D\longrightarrow \C^2$ a normal deformation  defined for all $1\leq t\leq 2$ so that  $$L_t\cap B_{\overline R/2}\subset F_t(D)\subset L_t\cap B_{\overline R}$$ and the $C^{2,\alpha}$ norm of $F_t$ is bounded by $K_1$.
\end{itemize}

\subsection{Main result}

\begin{thm}\label{final}Assume condition $(\star)$ holds. 

There are $\eta_0$ and $R_5,$  depending on $K_1$ and $\sigma$, so that if $\overline R\geq R_5$ in \eqref{defi2.geral} and 
$L_1$ is $\eta_0$-close in $C^{2,\alpha}(B_{R_5})$ to 
$$M_1=\{(\sigma(s)\cos \alpha, \sigma(s) \sin \alpha)\,|\, s\in [0,+\infty), \alpha\in S^1\}$$
 then $(L_t)_{t\geq 0}$ must have a singularity before $T_1=2A_1/ \pi+1$ (with $A_1$ defined in Section \ref{setup.sta}). 
\end{thm}
\begin{rmk}\label{motiv.rmk}
 The content of the theorem is that if $L_1$ is very close to $M_1$  and $\overline R$ sufficiently large, then the flow $(L_t)_{1\leq t\leq T_1}$ must have a finite time singularity.  The proof   proceeds by contradiction and we assume the existence of smooth flows $(L^i_t)_{0\leq t\leq T_1}$   with  $\overline R^i$ tending  to infinity and $L^i_1$ converging to  $M_1$ in $C^{2,\alpha}_{loc}$. Standard arguments show that  $(L^i_t)_{1\leq t\leq T_1}$ converges to  $(M_t)_{1\leq t\leq T_1}$ a (weak) solution to mean curvature flow starting at $M_1$. The rest of the argument will have two steps.

The first step, see Theorem \ref{sing.flow} ii)--iv), is to show the existence of a family of  curves $\sigma_t$ so that
$$M_t=\{(\sigma_t(s)\cos \alpha, \sigma_t(s) \sin \alpha)\,|\, s\in [0,+\infty), \alpha\in S^1\}$$ and show  that  $(\sigma_t)_{t\geq 1}$ behaves as depicted in Figure \ref{modelo-3} and Figure \ref{modelo-4}. More precisely, there is a singular time $T_0$ so that  $\sigma_t$ has  a single self-intersection  for  all $1\leq t<T_0$, $\sigma_{T_0}$ is embedded with a  singular point, and  $\sigma_t$ is an  embedded smooth curve  for $t>T_0$. Finally, and this will be important for the second step, we show in Theorem \ref{sing.flow} i) that $L^i_t$  converges in  $C^{2,\alpha}$ to $M_t$ in a small ball around the origin  and outside a large ball  for all $t\leq T_0+1$.

The second step (see details in  Corollary \ref{sing.equiv}) consists in considering the function
$$ f(t)=\theta_t(\infty)-\theta_t(0),$$
where $\theta_t(0)$ is the Lagrangian angle of $M_t$ at $0\in M_t$ and $\theta_t(\infty)$ is  the ``asymptotic'' Lagrangian angle of $M_t$ which makes sense because, due to   Lemma \ref{large}, $M_t$ is asymptotic to the  plane $P_1$. On one hand, because the curve $\sigma_t$ changes from a curve with a single self-intersection to a curve which is embedded as $t$ crosses $T_0$, we will see that
$$ \lim_{t\to T_0^-}f(t)=\lim_{t\to T_0^+}f(t)-2\pi.$$
On the other hand, because $L^i_t$ is smooth and converges to $M_t$ in small ball around the origin  and outside a large ball for all $t\leq T_0+1$, we will see that the function $f(t)$ is continuous. This gives us a contradiction.
\end{rmk}

\begin{proof}[Proof of Theorem \ref{final}] We argue by contradiction and assume the theorem  does not hold. In this case we can find $(L^i_t)_{0\leq t\leq T_1}$ a sequence of smooth flows which satisfies  condition $(\star)$ with  $\overline R^i$ tending  to infinity and $L^i_1$ converges to  $M_1$ in $C^{2,\alpha}_{loc}$.

Compactness for integral Brakke motions \cite[Section 7.1]{ilmanen1} implies that, after passing to a subsequence, $(L^i_t)_{0\leq t\leq T_1}$ converges to an integral Brakke motion $(M_t)_{0\leq t\leq T_1}.$ The next theorem characterizes $(M_t)_{0\leq t\leq T_1}.$

\begin{thm}\label{sing.flow}  There is $\delta_0$ small,  $r$ small, $R$ large, $T_0\in (1,T_1)$, and a continuous  family of curves $\sigma_t:[0,+\infty)\longrightarrow \C$ with 
$$\sigma_1=\sigma,\,\, \sigma_t^{-1}(0)=0\mbox{ for all } 1\leq t\leq T_0+\delta_0,$$ 
and such that
\begin{enumerate}
\item[i)]  For all $1\leq t\leq T_0+\delta_0$
\begin{itemize}
\item $M_t$ is smooth in $B_r\cup \C^2\setminus B_R$ and
\item  $L^i_t$ converges in $C_{loc}^{2,\alpha}$ to $M_t$ in $B_r\cup \C^2\setminus B_R$.
\end{itemize}
\item[ii)]For all $1\leq t<T_0$, $\sigma_t$ is a smooth curve  with a single self-intersection. Moreover
\begin{equation}\label{equiv.surface}
M_t=\{(\sigma_t(s)\cos \alpha, \sigma_t(s) \sin \alpha)\,|\, s\in [0,+\infty), \alpha\in S^1\}
\end{equation}
and
\begin{equation}\label{equiv.flow}
\frac{dx}{dt}=\vec k-\frac{x^{\bot}}{|x|^2}.
\end{equation}
Finally, for each $t<T_0$, $L^i_t$ converge in $C_{loc}^{2,\alpha}$ to $M_t$.
\item[iii)] The curve $\sigma_{T_0}$ has  a singular point $Q$ so that $\sigma_{T_0}\setminus\{Q\}$ consists of two disjoint smooth embedded arcs and, away from $Q$,  $\sigma_t$  converges to $\sigma_{T_0}$  as $t$ tends to $T_0$.
\item[iv)] For all $T_0<t\leq T_0+\delta_0$,  $\sigma_t$ is a smooth embedded curve  which satisfies \eqref{equiv.surface} and \eqref{equiv.flow}. Moreover,  for each $T_0<t\leq T_0+\delta_0,$  $L^i_t$ converge in $C_{loc}^{2,\alpha}$ to $M_t$.
\end{enumerate}
\end{thm}
\begin{rmk}
{\bf (1)} The content of this theorem is to justify  the behavior shown in Figure \ref{modelo-3} and Figure \ref{modelo-4}. More precisely, Theorem \ref{sing.flow} ii) and iii) say that the solution  $(\sigma_t)_{t\geq 1}$ to \eqref{equiv.flow}  with $\sigma_1=\sigma$ will have a singularity at time $T_0$ which corresponds to the loop enclosed by the self-intersection of $\sigma_t$ collapsing. Theorem \ref{sing.flow} iv) says that after $T_0$ the curves $\sigma_t$ become  smooth and embedded.

{\bf (2)} The behavior described above follows essentially from Angenent's work \cite{angenent0,angenent} on general one-dimensional curvature flows. 

{\bf (3)} We also remark that the fact $M_t$ has the symmetries described in  \eqref{equiv.surface} up to the singular time $T_0$   is no surprise because that is equivalent to uniqueness of solutions with smooth controlled data. After the singular time $T_0$ there is no general principle justifying why $M_t$ has the symmetries described in \eqref{equiv.surface}. The reason this occurs is because the function $\mu$ defined in Theorem \ref{annulus.estimates} v) evolves by the linear heat equation and is zero if and only if $M_t$  can be expressed as in \eqref{equiv.surface} (see Claim 1 in  proof of Theorem \ref{sing.flow} for details).

{\bf (4)} Theorem \ref{sing.flow} i) is necessary so that we can control the flow in neighborhood of the origin because the right-hand side of  \eqref{equiv.flow} is singular at the origin. It is important for Corollary \ref{sing.equiv} that the convergence mentioned in Theorem \ref{sing.flow} i) holds for all $t\leq T_0+\delta_0$ including the singular time.

{\bf (5)} The proof is mainly technical and will be given at the end of this section.
\end{rmk}

\begin{cor}\label{sing.equiv} Assuming Theorem \ref{sing.flow} we have that,  for all $i$ sufficiently large,  $(L^i_t)_{1\leq T_1}$ must have a finite time singularity.
\end{cor}
In Remark \ref{motiv.rmk} we sketched the idea behind the proof of this corollary.

\begin{proof}
	From Theorem \ref{sing.flow}  i) we can find a  small interval $I$ containing $T_0$ (the singular time of $\sigma_t$),  and pick $a_t \in \sigma_t\cap A(r/3,r/2)$, $b_t \in \sigma_t\cap A(2R,3R)$ so that $ a_t$, $b_t$ are the endpoints of a segment $\bar \sigma_t\subseteq \sigma_t\cap A(r/3,3R)$ and  the paths $(a_t)_{t\in I}$, $(b_t)_{t\in I}$ are smooth.  Consider the function
\begin{equation*}
f(t)=\theta_t(b_t)-\theta_t(a_t).
\end{equation*}
We claim that
\begin{equation}\label{f.cor}
\lim_{t\to T_0^-}f(t)=\lim_{t\to T_0^+}f(t)-2\pi.
\end{equation}
Recall  the Lagrangian angle $\theta_t$ equals, up to a constant, the argument of the complex number $\sigma_t\sigma_t'$. Hence, for all $t\in I\setminus\{T_0\},$ we have
	$$
	\theta_t(b_t)-\theta_t(a_t)=\int_{\bar\sigma_t}d\theta_t=\int_{\bar\sigma_t}\langle\vec k,\nu \rangle d\H^1-\int_{\bar\sigma_t}\left\langle\frac{x}{|x|^2},\nu \right \rangle d\H^1,
	$$
	where $\nu$ is the normal obtained by rotating the tangent vector to $\bar\sigma_t$ counterclockwise and we are assuming that this segment is oriented from $a_t$ to $b_t$. The curves $\bar\sigma_t$ are smooth near the endpoints by Theorem \ref{sing.flow} i), have a single self intersection for $t<T_0$ by Theorem \ref{sing.flow} ii), and are embedded for $t>T_0$  by Theorem \ref{sing.flow} ii) (see Figure \ref{modelo-4}). Thus the rotation index of $\sigma_t$ changes across $T_0$ and so 
		\begin{equation}\label{curve.equiv}
\lim_{t\to T_0^+} \int_{\bar\sigma_t}\langle\vec k,\nu \rangle d\H^1=\lim_{t\to T_0^-} \int_{\bar\sigma_t}\langle\vec k,\nu \rangle d\H^1+2\pi.
\end{equation}
	 The vector field $X=x|x|^{-2}$ is divergence free and so, because  none of the segments $\bar\sigma_t$   winds around the origin,  the Divergence Theorem implies
	 \begin{equation}\label{curve2.equiv}
\lim_{t\to T_0^+} \int_{\bar\sigma_t}\left\langle\frac{x}{|x|^2},\nu \right \rangle  d\H^1=\lim_{t\to T_0^-} \int_{\bar\sigma_t}\left\langle\frac{x}{|x|^2},\nu \right \rangle  d\H^1.
\end{equation}
Claim \eqref{f.cor} follows at once from \eqref{curve.equiv} and \eqref{curve2.equiv}.

From Theorem \ref{sing.flow} i) we can choose a sequence of  smooth paths $(a^i_t)_{t\in I}$, $(b^i_t)_{t\in I}$  converging to $(a_t)_{t\in I}$, $(b_t)_{t\in I}$ respectively, and such that  $a^i_t, b_t^i \in L^i_t$. Consider the function
	$$f^i(t)=\theta^i_t(b^i_t)-\theta^i_t(a^i_t).$$
For every $t\in I\setminus\{T_0\}$ we have from Theorem \ref{sing.flow} ii) and iv) that $L^i_t$ converges in $C_{loc}^{2,\alpha}$ to $M_t$. As a result, 
\begin{equation}\label{cont.equiv}
f_i(t)\mbox{ converges to }f(t)\mbox { for all }t\in I\setminus\{T_0\}.
\end{equation}
Because the flow $(L^i_t)_{t\in I}$ exists smoothly, the function $f^i(t)$ is smooth and 
$$\frac{df^i_t(t)}{dt}=\Delta \theta^i_t (b^i_t)+\langle\nabla \theta^i_t, db^i_t/dt\rangle-\Delta \theta^i_t (a^i_t)-\langle\nabla \theta^i_t, da^i_t/dt\rangle.$$
Hence, Theorem \ref{sing.flow} i) shows that $df^i(t)/dt$ is uniformly bounded (independently of $i$) for all $t\in I$. From \eqref{cont.equiv} we obtain that the function $f$ must be Lipschitz continuous and this contradicts \eqref{f.cor}.
\end{proof}

This corollary gives us the desired  contradiction and finishes the proof of the theorem.
\end{proof}

\subsection{Proof of  Theorem \ref{sing.flow}}

Recall the function $\mu=x_1y_2-y_1x_2$ defined in Theorem \ref{annulus.estimates} v). We start by proving two claims.
\vskip 0.05in
\noindent{\bf Claim 1:}
$
M_t \subseteq \mu^{-1}(0) \mbox{ and }|\nabla \mu|=0\mbox{ for almost all } 1\leq t \leq T_1.
$
\vskip 0.05in
From  Lemma \ref{huisken1} and Theorem  \ref{annulus.estimates} v) we have
\begin{multline}\label{bird.equiv}
\int_{L^i_t}\mu^2\Phi(0,1)d\H^2+\int_1^t\int_{L^i_s}|\nabla \mu|^2\Phi(0,1+t-s)d\H^2ds\leq \\
\int_{L^i_1}\mu^2\Phi(0,1+t)d\H^2+\int_1^t\int_{L^i_s}\left(\frac{|E|^2}{4}\mu^2+E_2\right)\Phi(0,1+t-s)d\H^2ds.
\end{multline}
Because $M_1\subseteq \mu^{-1}(0)$ and $E, E_2$ converge uniformly to zero when $i$ goes to infinity we obtain
\begin{multline*}
\lim_{i\to\infty}  \int_{L^i_1}\mu^2\Phi(0,1+t)d\H^2+\int_1^t\int_{L^i_s}\left(\frac{|E|^2}{4}\mu^2+E_2\right)\Phi(0,1+t-s)d\H^2ds\\
=\int_{M_1}\mu^2\Phi(0,1+t)d\H^2=0,
\end{multline*}
which combined with \eqref{bird.equiv} implies
$$\int_{M_t}\mu^2\Phi(0,1)d\H^2+\int_1^t\int_{M_s}|\nabla \mu|^2\Phi(0,1+t-s)d\H^2ds=0.$$
This proves the claim.
\vskip 0.05in
\noindent{\bf Claim 2:}  For every $\delta$ there is $R=R(\delta,T_1)$ so that, in the annular region $A(R,\overline R_i)$, $L_t^i$ is $\delta$-close in $C^{2,\alpha}$ to the plane $P_1$ for all $1\leq t\leq T_1$ and $i$ sufficiently large.
\vskip 0.05in
 According to Lemma \ref{large} there is a constant $R=R(\delta,T_1, \underline R_i)$ so that, in the annular region $A(R,\overline R_i)$, $L^i_t$  is $\delta$-close in $C^{2,\alpha}$ to $P_1$ for all $1\leq t\leq T_1$. Because   $L_1^i$  converges to $M_1$, we can deduce from Theorem \ref{connected} i) that $\underline R_i$ is bounded and thus the constant $R$ depends only on $\delta$ and $T_1$ and not on the index $i$. This prove  the claim.

\vskip 0.05in
\noindent{\bf Definition of ``singular time'' $T_0$: } First we need to introduce some notation. Because  condition $(\star)$ holds for the flow $(L^i_t)$, there are a sequence of discs $D_i$ of increasingly larger radius and normal deformations $F^i_t: D_i\longrightarrow \C^2$ so that, for all  $1\leq t\leq 2$, $F^i_t(D_i)\subseteq L^i_t$, and $F^i_t$ converges in $C^{2,\alpha}_{loc}$ to $F_t:\R^2\longrightarrow \C^2$, where $M_t=F_t(\R^2)$.

 Consider the following condition
\begin{equation}\label{extra.sta}
F^i_t\mbox{ converges in }C_{loc}^{2,\alpha}\mbox{ to } F_t:\R^2\longrightarrow \C^2,\mbox{ where }M_t=F_t(\R^2)
\end{equation}
 and set
 \begin{equation}\label{extra2.sta}
T_0=\sup\{l\,|\, F^i_t\mbox{ is defined and condition \eqref{extra.sta} holds for all }t\leq l\}\cap [1,T_1].
\end{equation}
\vskip 0.05in
\noindent{\bf Proof of Theorem \ref{sing.flow} ii): } By the way $T_0$ was chosen and Claim 1, we have that $M_t\subseteq \mu^{-1}(0)$ is a smooth surface diffeomorphic to $\R^2$. Thus Lemma \ref{ap3} implies the existence of $(\sigma_t)_{1\leq t <T_0}$ so that \eqref{equiv.surface} holds. Because $(M_t)_{1\leq t<T_0}$ is a smooth solution to mean curvature flow it is immediate to conclude \eqref{equiv.flow}.  From the definition of $T_0$ it is also straightforward to conclude that $L^i_t$ converges in $C_{loc}^{2,\alpha}$ to $M_t$ if $t<T_0$. We are left to argue that $\sigma_t$ has a single self-intersection for all $1\leq t <T_0$.
From Lemma \ref{equiv.sturm} below we conclude that if $\sigma_t$ develops a tangential self-intersection it must be away from the origin. It is easy to see from  the flow \eqref{equiv.flow} that cannot happen.

\begin{lemm}\label{equiv.sturm} That is $r$ so that $\sigma_t\cap B_r$ is embedded  for all $1\leq t <T_0$.
\end{lemm}
\begin{proof}
 Recall  the definition of $C_a$ in \eqref{equiv.cone}. We start by arguing that 
\begin{equation}\label{equiv.incl}
\sigma_t \subseteq C_a\mbox{  for all  }1\leq t <T_0.
\end{equation}
 The boundary of the cone $C_a$  consists of two half-lines which are fixed points for the flow  \eqref{equiv.flow}. From Claim 2 we see that $M_t$ is asymptotic to $P_1$ and so $\sigma_t$ does not intersect $\partial C_a$ outside  a large ball. Thus, because $\sigma_1\subset C_a$,  we conclude from Lemma \ref{equi.max}  that  $\sigma_t\subseteq C_a$ for all $1\leq t <T_0$.

Denote by $\Gamma$ a curve in $\C$  which is asymptotic at infinity to
\begin{equation}\label{equiv.asympto}
\{r\exp(i(\pi+3a/2))\,|\, r\geq 0\}\cup  \{r\exp(i(\pi/2+3a/2))\,|\, r\geq 0\}
\end{equation}
 and generates, under the $S^1$ action described in \eqref{equiv.surface},  a Special Lagrangian asymptotic to two planes (Lawlor Neck). In particular, the curves $\Gamma_{\delta}=\delta\Gamma$ are fixed points for the flow \eqref{equiv.flow} for all $\delta$ and, because of \eqref{equiv.incl} and \eqref{equiv.asympto}, $\sigma_t$  does not intersect $\Gamma_{\delta}$ outside a large ball for all $1\leq t <T_0$.

  From the description of $\sigma$ given at the beginning of Section \ref{sta}, we find  $\delta_0$ so that for every $\delta<\delta_0$ the curve $\Gamma_{\delta}$ intersects  $\sigma$ only once. Hence,  we can apply \cite[Variation on Theorem 1.3]{angenent}  and conclude that $\Gamma_{\delta}$ and $\sigma_t$  intersect only once  for  all $1\leq t<T_0$ and all $\delta<\delta_0$. It is simple to see that this implies the result we want to show provided we choose $r$ small enough.
\end{proof}

\noindent{\bf Proof of Theorem \ref{sing.flow} i):} This follows from Claim 2 and the next lemma.
\begin{lemm}\label{small-ball}
There are $r$ small and $\delta$ small so that $M_t\cap B_r$ is smooth, embedded, and $L^i_t$ converges in $C^{2,\alpha}(B_r)$ to $M_t\cap B_r$ for all $1\leq t\leq T_0+\delta$.

In particular, the curve $\sigma_t\cup -\sigma_t$ is smooth and embedded near the origin with bounds on its $C^{2,\alpha}$ norm for all $1\leq t\leq T_0+\delta$.
\end{lemm}
\begin{rmk}\label{rmk.small-ball}   
 The key step to show Lemma \ref{small-ball} is to argue that $(M_t)_{t\geq 1}$ develops no singularity at the origin at time $T_0$ and the idea is the following. First principles will show that a sequence of of blow-ups  at the origin $(\sigma^j_t)_{s<0}$ of $(\sigma_t)_{t<T_0}$  converge in $C_{loc}^{1,1/2}(\R^2-\{0\})$ to a union of half-lines. But Lemma \eqref{equiv.sturm} implies $\sigma^j_t$ is embedded in $B_1$ for all $j$ sufficiently large and so it must converge to a single half-line. White's Regularity Theorem implies  no singularity occurs.
\end{rmk}

\begin{proof} From the way $T_0$ was chosen \eqref{extra2.sta} and Lemma \ref{equiv.sturm} we know the existence of $r$ so that   $M_t\cap B_r$ is smooth, embedded, and $L^i_t$ converges in $C^{2,\alpha}(B_r)$ to $M_t\cap B_r$ for all $1\leq t<T_0$. To extend this to hold up to $t=T_0$ (with possible smaller $r$) it suffices to show that  $(M_t)_{1\leq t< T_0}$  develops no  singularity at the origin at time $T_0$. 

Choose a sequence $(\lambda_j)_{j \in \N}$ tending to infinity and set
	$$M^j_t=\lambda_j M_{T_0+t/\lambda_j^2},\quad\mbox{for all } t<0.$$ 
From \cite[ Lemma 5.4]{neves} we have the existence of a union of planes $Q$ with support contained in  $\mu^{-1}(0)$ such that, after passing to a subsequence and for almost all $t<0$, $M^j_t$ converges in the varifold sense to $Q$ and 
\begin{equation}\label{H}
\lim_{j\to\infty}\int_{M^j_t}(|H|^2+|x^{\bot}|^2)\exp({-|x|^2})d\H^2=0.
\end{equation}
From \eqref{sing.flow} we can find curves $\sigma^j_t$ so that
$$M^j_t=\{(\sigma^j_t(s)\cos \alpha, \sigma^j_t(s) \sin \alpha)\,|\, s\in [0,+\infty), \alpha\in S^1\}.$$
We obtain from \eqref{H} that for almost all $t$ and every $0<\eta<1$
$$\lim_{j\to\infty}\int_{\sigma^j_t\cap A(\eta,\eta^{-1})}|\vec k|^2+|x^{\bot}|^2d\H^1=0,$$
which implies that $\sigma^j_t$ converges in $C_{loc}^{1,1/2}(\R^2-\{0\})$ to a union of half-lines with endpoints at the origin. Lemma \ref{equiv.sturm} implies that for all $j$ sufficiently large $\sigma^j_t$ is embedded inside the unit ball. Thus $\sigma^j_t$ must converge to a single half-line and so $Q$ is a multiplicity one plane.  Thus there can be no singularity at time $T_0$ at the origin.

 We now finish the proof of the lemma. So far we have proven that $M_t$ is smooth and embedded near the origin for all $1\leq t\leq T_0$ .Thus we can find $l_0$  small so that
$$\int_{M_t}\Phi(x,l)d\H^2\leq 1+{\varepsilon_0}\mbox{ for every }x\in B_{2l_0}, l\leq 4l_0^2,\mbox{ and }1\leq t\leq T_0. $$
Monotonicity formula implies that
  $$\int_{M_t}\Phi(x,l)d\H^2\leq 1+{\varepsilon_0}\mbox{ for every }x\in B_{l_0}, l\leq l_0^2,\mbox{ and }1\leq t\leq T_0+l_0^2.$$
Because $L^i_t$ converges to $M_t$ as Radon measures, White's Regularity Theorem implies uniform $C^{2,\alpha}$ bounds in $B_{l_0/2}$ for $L^i_t$ whenever $i$ is sufficiently large and $t\leq T_0+l_0^2.$ The lemma follows then straightforwardly.
\end{proof}

\noindent{\bf Proof of Theorem \ref{sing.flow} iii):} We need two lemmas first.
\begin{lemm}\label{cults} $T_0<T_1.$
\end{lemm}
\begin{rmk}
The idea is to show that if $T_0\geq T_1$, then the loop of $\sigma_t$ created by its self-intersection would have negative area.
\end{rmk}
\begin{proof}
 Suppose $T_0=T_1$.  Denote by $q_t$ the single self-intersection of $\sigma_t$, by $c_t\subseteq \sigma_t$ the closed loop  with endpoint $q_t$, by $\alpha_t \in [-\pi,\pi]$ the exterior angle that $c_t$ has at the vertex $q_t$, by $\nu$ the interior unit normal, and  by $A_t$ the area enclosed by the loop.  From Gauss-Bonnet Theorem we have
$$\int_{c_t}\langle\vec k,\nu\rangle d\H^1+\alpha_t=2\pi\implies \int_{c_t}\langle\vec k,\nu\rangle d\H^1\geq \pi.$$
A standard formula shows that 
$$\frac{d}{dt}A_t=-\int_{c_t}\left \langle \vec k-\frac{{x^{\bot}}}{|x|^{2}},\nu\right\rangle d\H^1\leq -\pi+\int_{c_t}\left \langle\frac{{x}}{{|x|^{2}}},\nu\right\rangle d\H^1=-\pi,$$
where the last identity follows from the Divergence Theorem combined with the  fact that $c_t$ does not contain the origin in its interior. Hence  $0\leq A_t\leq A_1-(t-1)\pi$ and  making $t$ tending to $T_1=2A_1/\pi+1$ we obtain a contradiction.
\end{proof} 

\begin{lemm}\label{technical}
The curve $\sigma_t$ must become singular when $t$ tends to $T_0$.
\end{lemm}
\begin{rmk}
The flow $(M_t)_{t\geq 0}$ is only a weak solution to mean curvature flow which means that, in principle,  $\sigma_{T_0}$ could be a smooth curve with a self intersection and, right after, $\sigma_t$ could split-off the self intersection and become instantaneously a disjoint union of a circle with a half-line. This lemma shows that, because $M_t$ is a limit of smooth flows $L^i_t$, this phenomenon cannot happen. The proof is merely technical.
\end{rmk} 
\begin{proof}
We are assuming $F^i_{t}$ converges in $C^{2,\alpha}_{loc}$ to $F_t$ for all $t<{T_0}$. Assuming $\sigma_{T_0}$ is smooth we have from parabolic regularity that $(\sigma_t)_{t\leq {T_0}}$ is a smooth flow. Thus $M_{T_0}$ is also smooth and  the maps $F_t$ converge smoothly to a map $F_{T_0}:\R^2\longrightarrow \C^2$. Therefore, there is a constant $C$ which bounds the $C^2$ norm of $F_t$ for all ${T_0}-1\leq t\leq {T_0}$. Hence, using Claim 2 to control the $C^{2,\alpha}$ norm of $F^i_t$ outside a large ball, we obtain that for $\bar t<{T_0}$ and $i$ sufficiently large, the $C^{2}$ norm of $F^i_{\bar t}$ is bounded by $2C$. Looking at the evolution equation of $|A|^2$ it is then a standard application of the maximum principle to find $\delta=\delta(C)$ such that  the second fundamental form of the immersion $F^i_{t}$ is bounded by $4C$ for all $\bar t\leq t\leq \bar t+\delta$. Therefore, choosing $\bar t$ such that ${T_0}<\bar t+\delta$, parabolic regularity implies condition \eqref{extra.sta}  holds for all $t$ slightly larger than ${T_0}$ which, due to Lemma \ref{cults}, contradicts the maximality of ${T_0}$.
\end{proof}

Claim 2 and Lemma \ref{small-ball} give us control of the flow \eqref{equiv.flow} outside an annulus. Hence we apply Theorem \ref{angenent} and conclude the singular curve $\sigma_{T_0}$ contains a point $Q$ distinct from the origin such that $\sigma_{T_0}\setminus\{Q\}$ consists of two smooth disjoint arcs and, away from the singular point, the curves $\sigma_t$ converge smoothly to $\sigma_{T_0}$ (see Figure \ref{modelo-4}).  
\vskip 0.05in
\noindent{\bf Proof of Theorem \ref{sing.flow} iv):} 
From Theorem \ref{sing.flow} i) and Claim 1, we can apply Lemma \ref{ap3} and conclude that $M_t$ can be described by a one dimensional varifold $\sigma_t\subset \C$ for almost all $T_0<t<T_1$.

In \cite[Section 8]{angenent0} Angenent constructed an embedded smooth one solution $(\gamma_t)_{t>0}$ which tends to $\sigma_{T_0}$ when $t$ tends to zero and which looks like the solution described on Figure \ref{modelo-4}. The next lemma is the key to show Theorem \ref{sing.flow} iv).
 
\begin{lemm}\label{fatlemma}There is $\delta$ small so that $\gamma_t=\sigma_{{T_0}+t}$  for all $0<t <\delta$.
\end{lemm}
\begin{rmk}
  This lemma amounts to show that  there is a unique (weak) solution to the flow \eqref{equiv.flow} which starts at $\sigma_{T_0}$.

The idea to prove this lemma, which we now sketch, is well known among the specialists. Consider  $\gamma^i_+,\gamma^i_-$  two sequences of smooth embedded curves  with an endpoint at  the origin and converging to $\sigma_{T_0}$, with $\gamma^i_+,\gamma^i_-$ lying above and below $\sigma_{T_0}$, respectively. There is a region $A_i$ which has $\sigma_{T_0} \subseteq A_i$ and $\partial A_i=\gamma^i_+\cup\gamma^i_-.$ Denote the flows starting  at $\gamma^i_{+}$ and $\gamma^i_-$ by $\gamma^i_{+,t}$ and  $\gamma^i_{-,t}$ respectively, and  use $A_i(t)$ to denote the region below $\gamma^i_{+,t}$ and above $\gamma^i_{-,t}$.

For the sake of the argument we can assume that $A_i$ is finite and tends to zero when $i$ tends to infinity. A simple computation will show that $\mbox{area}(A_i(t))\leq \mbox{area}(A_i)$ and so, like $A_i$,  the area of $A_i(t)$ tends to zero when $i$ tends to infinity. The avoidance principle for the flow implies that $\sigma_{{T_0}+t}, \gamma_t\subseteq A_i(t)$ for all $i$ and $t$, and thus, making $i$ tend to infinity, we obtain that  $\sigma_{{T_0}+t}=\gamma_t$.

The proof requires some technical work to go around  the fact  the curves $\gamma^i_+,\gamma^i_-$ are non compact and thus $A_i$ could be infinity.
\end{rmk}

\begin{proof}
Let $\gamma^i_+,\gamma^i_-:[0,+\infty]\longrightarrow \C$ be two sequences of smooth embedded curves converging to $\sigma_{T_0}$ with $\gamma^i_+,\gamma^i_-$ lying above (below) $\sigma_{T_0}$ and such that 
\begin{equation}\label{pm.condition}
(\gamma_{\pm}^{i})^{-1}(0)=0,\quad  \gamma_{\pm}^i\cup -\gamma_{\pm}^i \mbox{ is smooth,}\quad \theta^i_+(0)<\theta_{T_0}(0)<\theta^i_-(0).
\end{equation}
The convergence is assumed to be strong on compact sets not containing the cusp point of $\sigma_{T_0}$.  Denote by $\gamma^i_{\pm,t}$ the solution to the equivariant flow  \eqref{equiv.flow} with initial condition $\gamma^i_{\pm}$. Short time existence was proven in \cite[Section 4]{neves} provided we assume controlled behavior at infinity.  The same arguments  used to study $\sigma_t$ (namely Lemma \ref{small-ball}) show that embeddeness is preserved and no singularity of   $\gamma^i_{\pm,t}$ can occur at the origin. Hence an immediate consequence of Theorem \ref{angenent}   is that the flow exists smoothly for all time. 

From  the last condition in \eqref{pm.condition} we know $\gamma^i_+$ intersects transversely  $\gamma^i_-$ at the origin. Furthermore we can choose $\gamma^i_+,\gamma^i_-$ to be not asymptotic to each other at infinity. Thus we can apply Lemma \ref{equi.max} and conclude that  $\gamma^i_{+,t}$  and $\gamma^i_{-,t}$ intersect each other  only at the origin. Hence there is an open region $A_i(t)\subset \C$  so that $\gamma^i_{+,t}\cup \gamma^i_{-,t}=\partial A_i(t)$.

From Claim 2  we know that $\gamma_{T_0}$ is asymptotic to a straight line. Thus we can reason like in the proof of Theorem \ref{annulus.estimates} i) and conclude the existence of  $R_i$ tending to infinity  so that  $\gamma^i_{\pm, t}\cap A(R_i/2,2R_i)$ is graphical over the   real axis with $C^{1}$ norm  smaller than $1/i$ for all $0\leq t\leq 1$.  

Consider $B_i(t)=A_i(t)\cap \{(x,y)\,|\, x\geq -R_i\}$. This region has the origin as one of its ``vertices'' and is bounded by three smooth curves. The top  curve is part  of $\gamma^i_{+,t}$, the bottom curve is part of $\gamma^i_{-,t}$, and left-side curve is part of $\{x=-R_i\}$. 
Using the fact that
$$\mbox{area}(B_i(t))=\int_{\partial B_i(t)} \lambda,$$
differentiation shows that
$$
\frac{d}{dt}\mbox{area}(B_i(t))  =-\left(\theta^i_{+,t}(-R_i)-\theta^i_{-,t}(-R_i)\right)+\left(\theta^i_{+,t}(0)-\theta^i_{-,t}(0)\right),
$$
where $\theta^i_{\pm,t}(-R_i)$ denote the Lagrangian angle of $\gamma^i_{\pm,t}$ at the intersection with  $\{x=-R_i\}$ and $\theta^i_{\pm,t}(0)$ denotes the Lagrangian angle of $\gamma^i_{\pm,t}$ at the origin. Because $\gamma^i_{+,t}$ lies above $\gamma^i_{-,t}$ and they intersect at the origin, we have $\theta^i_{+,t}(0)\leq \theta^i_{-,t}(0)$. Thus
$$
\frac{d}{dt}\mbox{area}(B_i(t))  \leq -\left(\theta^i_{+,t}(-R_i)-\theta^i_{-,t}(-R_i)\right).
$$
Recalling that $\gamma^i_{\pm, t}\cap A(R_i/2,2R_i)$ is graphical over the   real axis with $C^{1}$ norm  smaller than $1/i$ for all $t\leq 1$, we have that the term on the  right side of the above inequality    tends to zero when $i$ tends to infinity. Finally the curves can be chosen so that $\mbox{area}(B_i(0))\leq 1/i$ and thus 
\begin{equation}\label{fat}
\lim_{i\to\infty}\area(B_i(t))\leq\lim_{i\to\infty}\area(B_i(0))=0.
\end{equation}

We now argue the existence of $\delta$ so that
\begin{equation}\label{inclusion}
 \gamma_t,\sigma_{{T_0}+t} \subseteq A_i(t)\mbox{ for all }t\leq \delta \mbox{ and all }i. 
\end{equation}
The inclusion for $\gamma_t$ follows from  Lemma \ref{equi.max}. Next we want to deduce the inclusion for the varifolds $\sigma_{{T_0}+t}$ (recall Lemma \ref{ap3}) which does not follow directly from Lemma \ref{equi.max} because $\sigma_{T_0+t}$ might not be smooth. We remark that the right-hand side of \eqref{equiv.flow} is the geodesic curvature with respect to the metric $h=(x_1^2+y_1^2)(dx_1^2+dy_1^2)$.  Because $(M_t)_{t\geq 1}$ is a Brakke flow, it is not hard to deduce from Lemma \ref{ap3} that $(\sigma_t)_{t\geq 1}$ is also a Brakke flow with respect to the metric $h$. This metric is singular at the origin and has unbounded curvature but fortunately, due to Claim 2 and Lemma \ref{small-ball},  we already know that $\sigma_t$ is smooth in a neighborhood of the origin and outside a compact set for all $t\leq {T_0}+\delta$. Thus,  the Inclusion Theorem proven in \cite[10.7 Inclusion Theorem]{ilmanen1} adapts straightforwardly  to our setting and this implies $\sigma_{{T_0}+t} \subseteq A_i(t)$ for all $t\leq \delta$.

Combining \eqref{fat} with \eqref{inclusion} we obtain that $\gamma_t=\sigma_{{T_0}+t}$ all $0<t<\delta$.

 \end{proof}

From Lemma  \eqref{fatlemma} we obtain that $M_t$ is smooth, embedded, and satisfies \eqref{equiv.surface}, \eqref{equiv.flow} for all $T_0<t<T_0+\delta$. Finally, from the fact that $M_t$ is embedded, it follows in a straightforward manner from White's Regularity Theorem that $L^i_t$ converges in $C_{loc}^{2,\alpha}$ to $M_t$. This completes the proof of Theorem \ref{sing.flow}

\section{Main Theorem}\label{ultima}

\begin{thm}\label{main}
For any embedded closed Lagrangian surface $\Sigma$ in $M$, there is $L$ Lagrangian in the same Hamiltonian isotopy class so that the Lagrangian mean curvature flow with initial condition $L$ develops a finite time singularity.
\end{thm}
\begin{proof}

{\bf Setup:}    Given $\overline R$  large we can find a  metric $g_R=R^2g$ (see Section \ref{prelim}) so that the hypothesis on ambient space described in Section \ref{ambient.space} are satisfied. Pick $p\in \Sigma$ and assume the Darboux  chart $\phi$ sends the origin into $p\in \Sigma$ and  $T_p\Sigma$ coincides with the real plane  $\R\oplus i\R\subseteq \C^2$ oriented positively.

 We can assume  $\Sigma\cap B_{4\overline R}$  is given by  the graph of the gradient of  some function defined over  the real plane, where the $C^2$ norm can be made arbitrarily small. It is simple to find $\overline \Sigma$  Hamiltonian isotopic to $\Sigma$ which coincides with the real plane in $B_{3\overline R}$. Denote by $L$ the Lagrangian which is obtained by replacing $\overline \Sigma\cap B_{3\overline R}$ with $N(\varepsilon,\underline R)$ defined in \eqref{basic.block.eq}. Using \cite[Theorem 1.1.A]{yasha}  we obtain at once that $L$ is Hamiltonian isotopic to $\overline \Sigma$ and hence to $\Sigma$ as well. Moreover, there is $K_0$ depending only on $\Sigma$  so that the hypothesis on $L$ described  in Section \ref{hyp.lagrangian} are  satisfied for all $\overline R$ large.
 
 We recall once more that $L$ depends on $\varepsilon, \underline R$, $\overline R$, and that $\overline R\geq 4\underline R$.  Assume the Lagrangian mean curvature flow $(L_t)_{t\geq 0}$ with initial condition $L$ exists smoothly for all time.
 
 \vskip 0.05in
{\bf First Step:}  Pick $\nu_0$ small (to be fixed later) and  choose  $\varepsilon,$ $\underline R$, and $\overline R$  so that Theorem \ref{connected} (with $\nu=\nu_0$) and Theorem \ref{annulus.estimates} hold. Thus, there is $R_1=R_1(\nu_0,K_0)$ so that, see Theorem \ref{connected} ii),
\begin{itemize}
\item[(A)] for every $1\leq t\leq 2$, $L_t\cap A(R_1, \overline R)$ is $\nu_0$-close in $C^{2,\alpha}$ to $L$.
\end{itemize}
Moreover, from Theorem \ref{connected}  iii) and iv), $L_t\cap B_{R_1}$ is contained in two connected components $Q_{1,t}\cup Q_{2,t}$ where
\begin{itemize}
\item[(B)] for every $1\leq t\leq 2$,  $Q_{2,t}$ is $\nu_0$-close in $C^{2,\alpha}(B_{R_1})$ to $P_3$.
\end{itemize}
 \vskip 0.05in
{\bf Second  Step:} We need to control $Q_{1,t}$. Apply Theorem \ref{proximity.final} with $S_0=R_1$ and $\nu=\nu_0$. Thus for all $\varepsilon$ small and $\underline R$ large we have that
\begin{itemize}
\item[(C)] for every $1\leq t\leq 2$,  $t^{-1/2}Q_{1,t}$ is $\nu_0$-close in $C^{2,\alpha}(B_{R_1})$ to $\mathcal{S}$,
\end{itemize}
where $\mathcal{S}$ is the self-expander defined in \eqref{Q} (see Figure \ref{modelo-se}).  
One immediate consequence of (A), (B), and (C) is the existence of $K_1$ so that for all $\varepsilon$ small and $\underline R$ large we have
\begin{itemize}
\item[($\star\star$)]the existence of a disc  $D$, and $F_t: D\longrightarrow \C^2$ a normal deformation  defined for all $1\leq t\leq 2$, so that  $$L_t\cap B_{\overline R/2}\subset F_t(D)\subset L_t\cap B_{\overline R}$$ and the $C^{2,\alpha}$ norm of $F_t$ is bounded by $K_1$.
\end{itemize}
 \vskip 0.05in
{\bf Third  Step:} Fix $\varepsilon$ and $\underline R$ in the definition of $L$ so that (A), (B), (C), and $(\star\star)$ hold, but let $\overline R$ tend to infinity. We then obtain a sequence of smooth flows   $(L^i_t)_{t\geq 0}$,  where $L^i_0$ converges strongly to $N(\varepsilon,\underline R)$ defined in \eqref{basic.block.eq} (see Figure \ref{modelo}). 
\begin{lemm}\label{guns}
 If $\nu_0$ is chosen small enough, there is a curve $\sigma\subset \C$ with all the properties described in Section \ref{setup.sta} (see Figure \ref{modelo-3}) and such that $L^i_1$ tends in $C^{2,\alpha}_{loc}$ to
 \begin{equation}\label{nunca}
 M_1=\{(\sigma(s)\cos \alpha, \sigma(s)\sin \alpha)\,|\, \alpha \in S^1, s\in [0,+\infty)\}.
 \end{equation}
\end{lemm}
Assuming this lemma  we will show that $(L^i_t)_{t\geq 0}$ must have a singularity for all $i$ sufficiently large which finishes the proof of Theorem \ref{main}. 
Indeed, because the flow $(L^i_t)_{t\geq 0}$ has property $(\star\star)$, we have at once that condition $(\star)$ of Section \ref{setup.sta} is satisfied. Hence Lemma \ref{guns} implies that $L^i_1$ satisfies the hypothesis of Theorem \ref{final}  for all $i$ sufficiently large and thus Theorem \ref{final} implies that  $(L^i_t)_{t\geq 0}$ must have  a finite time singularity.

\begin{proof}[Proof of Lemma \ref{guns}]
From condition $(\star\star)$  we have that $L^i_1$ converges in $C^{2,\alpha}_{loc}$ to a smooth Lagrangian $M_1$ diffeomorphic to $\R^2$. Moreover, from (B) and (C) we see that we can choose $\nu_0$ small so that $M_1$ is embedded in a small neighborhood the origin.  We argue that $M_{1}\subset \mu^{-1}(0)$, where the function $\mu=x_1y_2-y_1x_2$ was defined in Theorem \ref{annulus.estimates} v). 
From  Lemma \ref{huisken1} and Theorem  \ref{annulus.estimates} v) we have
\begin{multline}\label{bird2.equiv}
\int_{L^i_1}\mu^2\Phi(0,1)d\H^2+\int_0^1\int_{L^i_t}|\nabla \mu|^2\Phi(0,2-t)d\H^2dt\leq \\
\int_{L^i_0}\mu^2\Phi(0,2)d\H^2+\int_0^1\int_{L^i_t}\left(\frac{|E|^2}{4}\mu^2+E_2\right)\Phi(0,2-t)d\H^2dt.
\end{multline}
The terms $E, E_2$ converge uniformly to zero when $i$ goes to infinity because the ambient metric converges to the Euclidean one. Moreover $N(\varepsilon,\underline R)\subset \mu^{-1}(0)$ and so we obtain from \eqref{bird2.equiv} that
\begin{multline*}
\int_{M_1}\mu^2\Phi(0,1)d\H^2=\lim_{i\to\infty}  \int_{L^i_1}\mu^2\Phi(0,1)d\H^2\\
\leq \lim_{i\to\infty}\int_{L^i_0}\mu^2\Phi(0,2)d\H^2+\int_0^1\int_{L^i_t}\left(\frac{|E|^2}{4}\mu^2+E_2\right)\Phi(0,2-t)d\H^2dt\\
=\int_{N(\varepsilon,\underline R)}\mu^2\Phi(0,2)d\H^2=0.
\end{multline*}
Hence,   $M_{1}\subset \mu^{-1}(0)$ and we can apply Lemma \ref{ap3} to conclude the existence of a curve $\sigma$ so that \eqref{nunca} holds.

In order to check that $\sigma$ has the properties described in Section \ref{setup.sta} it suffices to see that $\sigma$ has a single self-intersection and is contained in the cone $C_a$ (defined in \eqref{equiv.cone}) because the remaining properties follow  from $M_1$ being diffeomorphic to $\R^2$, embedded near the origin, and asymptotic to the plane $P_1$ ( Lemma \ref{large}).

Recall that $\gamma(\varepsilon,\underline R)$, $\chi,$ and  $c_3$ are the curves in $\C$ which define, respectively, the Lagrangian $N(\varepsilon,\underline R)$, the self-expander $\mathcal{S}$, and the plane $P_3$.  Now $M_1$, being the limit of $L^i_1$, also satisfies (A), (B), and (C). Hence we know that $\sigma$ is $\nu_0$-close in $C^{2,\alpha}$ to $\gamma(\varepsilon,\underline R)$ in $\C\setminus B_{R_1}$  and that $\sigma\cap B_{R_1}$ has two connected components, one  $\nu_0$-close in $C^{2,\alpha}(B_{R_1})$  to $c_3$ and the other  $\nu_0$-close in $C^{2,\alpha}(B_{R_1})$  to $\chi$. It is simple to see that if $\nu_0$ is small then indeed all the desired properties for $\sigma$ follow.
\end{proof}

 \end{proof}

\section{Appendix}\label{appendix}

\subsection{Lagrangians with symmetries}
Recall that $\mu(x_1,y_1,x_2,y_2)=x_1y_2-x_2y_1$ and consider two distinct conditions on $M$. 
\begin{itemize}
\item[C1)] $M$ is an integral Lagrangian varifold which is a smooth embedded surface in a neighborhood of the origin;
\item[C2)] There is a smooth Lagrangian immersion $F:\R^2 \longrightarrow \C^2$ so that $M=F(\R^2)$ and $M$ is a smooth embedded surface in a neighborhood of the origin.
\end{itemize}

\begin{lemm}\label{ap3} Assume  $M\subseteq \mu^{-1}(0)$.

If C1) holds then there is a one-dimensional integral varifold $\gamma\subset \C$ so that for every function $\phi$ with compact support
\begin{equation}\label{vari.ap}
\int_M \phi d\H^2=\int_{\gamma} |z|\int_0^{2\pi}\phi(z\cos\alpha, z\sin\alpha)d\alpha d\H^1.
\end{equation}

If C2) holds then 
there is a smooth immersed curve $\gamma:[0,\infty)\longrightarrow \C$ with $\gamma^{-1}(0)=0$, and 
\begin{equation}\label{smooth.ap}M=\{(\gamma(s)\cos \alpha,\gamma(s)\sin \alpha)\,|\, s\in [0,+\infty), \theta \in S^1\}.
\end{equation}

In both cases  the curve (or varifold) $\gamma\cup-\gamma$ is smooth near the origin.
\end{lemm}
\begin{proof}
Consider the vector field
$$ X=-JD\mu=(-x_2,-y_2,x_1,y_1).$$
A simple computation shows that for any Lagrangian plane $P$ with orthonormal basis $\{e_1,e_2\}$ we have
\begin{equation}\label{div.ap}
\mbox {div}_P X=\sum_{i=1}^{2}\langle D_{e_i}X,e_i\rangle=0.
\end{equation}
Finally, consider $(F_{\alpha})_{\alpha \in S^1}$  to be the one parameter family of diffeomorphisms  in $SU(2)$ such that
$$\frac{dF_{\alpha}}{d\alpha}(x)=X(F_{\alpha}(x))
 $$
Consider the functions $f_1(x)=\arctan(x_2/x_1)$ and $f_2(x)=\arctan(y_2/y_1)$ which are defined, respectively, in $U_1=\{x_1\neq 0\}$,  $U_2=\{y_1\neq 0\}$.
Assume that C1) holds. We now make several remarks which will be important when one applies the co-area formula.

First,  $F_{\alpha}(M)=M$, i.e.,
$$\int_{M}\phi\circ F_{\alpha}d\H^2=\int_M \phi d\H^2\mbox{ for all }\phi\mbox{ with compact support.}
$$
Because $M\subseteq \mu^{-1}(0)$ and $M$ is Lagrangian we have that $X$ is a tangent vector to $M$. Hence
$$
\frac{d}{d\alpha}\int_{M}\phi\circ F_{\alpha}d\H^2=\int_M\langle D(\phi\circ F_{\alpha}), X\rangle d\H^2\\
=-\int_{M}(\phi\circ F_{\alpha}) \mbox{div}_M X d\H^2=0,
$$
where the last identity follows from \eqref{div.ap}.

Second, on $M$ we have $|\nabla f_i|(x)=|x|^{-1}$. Indeed, for every $x\in U_i$ with $\mu(x)=0$ it is a simple computation to see that $Df_i\in \mbox{span}\{X(x),JX(x)\}$ and thus, because $X$ is a tangent vector,
$$|\nabla f_i|(x)= |\langle Df_i(x), X(x)\rangle||x|^{-1}=|x|^{-1}.$$

Third, for almost all $\alpha$ and $i=1,2$,  $f_i^{-1}(\alpha)\cap M$ is a one-dimensional varifold. Moreover,  a simple computation shows $f_i\circ F_{\alpha}(x)=\alpha+f_i(x)$ for all $x\in U_i$ and all $\alpha\in(-\pi/2,\pi/2)$ and thus $$F_{\alpha}(f_i^{-1}(0)\cap M)=f_i^{-1}(\alpha)\cap F_{\alpha}(M)=f_i^{-1}(\alpha)\cap M.$$

Fourth, the fact that $M\subseteq \mu^{-1}(0)$ implies that $f_i^{-1}(0)\cap M$ has support contained in $\{x_2=y_2=0\}=\C$. Moreover, $f_1=f_2$ on $M$ and so  we set $$\Gamma=f_1^{-1}(0)\cap M=f_2^{-1}(0)\cap M.$$

Fifth, one can check that $F_{\pi}$ coincides with the antipodal map $A$. Thus $$A(\Gamma)=A(f_i^{-1}(0))\cap A(M)=f_i^{-1}(0)\cap M=\Gamma.$$
As a result, there is a one-dimensional varifold $\gamma$ such that $\Gamma=A(\gamma)+\gamma$ (the choice of $\gamma$ is not unique).

Finally we can apply the co-area formula and obtain for every $\phi$ with compact support in $U_i$
\begin{align*}
\int_M \phi d\H^2 & = \int_{-\pi/2}^{\pi/2}\int_{f_i^{-1}(\alpha)\cap M}\frac{\phi}{|\nabla f_i|}d\H^1d\alpha=\int_{-\pi/2}^{\pi/2}\int_{f_i^{-1}(\alpha)\cap M}|x|{\phi}d\H^1d\alpha\\
& = \int_{-\pi/2}^{\pi/2}\int_{f_i^{-1}(0)\cap M}|F_{\alpha}(x)|{(\phi\circ F_{\alpha})}d\H^1d\alpha\\
& = \int_{\Gamma}|x|\int_{-\pi/2}^{\pi/2}{\phi(F_{\alpha}(x))}d\alpha d\H^1\\
&= \int_{\Gamma} |z|\int_{-\pi/2}^{\pi/2}\phi(z\cos\alpha, z\sin\alpha)d\alpha d\H^1\\
&=\int_{\gamma} |z|\int_0^{2\pi}\phi(z\cos\alpha, z\sin\alpha)d\alpha d\H^1.
\end{align*}
This proves \eqref{vari.ap} for functions with support contained in $U_i$. Because $M$ is smooth and embedded near the origin it is straightforward to extend that formula to all functions with compact support.

Assume that C2) holds. From what we have done it is straightforward to obtain the existence of a curve $\gamma:I\longrightarrow \C$, where $I$ is a union of intervals, so that \eqref{smooth.ap} holds.
The fact that $M$ is diffeomorphic to $\R^2$ implies that $\gamma$ is connected and that $\gamma^{-1}(0)$ must be nonempty. The condition that $M$ is embedded when restricted to a small neighborhood of the origin implies that $\gamma^{-1}(0)$ must have only one element which we set to be zero. Finally, the fact that the map $F$ is an immersion is equivalent to the  curve  $\gamma\cup-\gamma$ being smooth at the origin.
\end{proof}

\subsection{Regularity for equivariant flow}\label{reg.eq} 

Angenent in \cite{angenent0} and \cite{angenent} developed the regularity theory for a large class of  parabolic  flows of curves in surfaces. We collect  the necessary  results, along with an improvement done in \cite{oaks}, which will be used in our setting. 

Let $\gamma_t:[0,a]\longrightarrow \C$, $0\leq t<T,$ be a one parameter family of smooth curves so that
\begin{itemize}
\item[A1)]There is $r>0$  and $p\in \C$ so that for all $0\leq t<T$, $\gamma_t(0)=0$, $\gamma_t(a)\in B_{r}(p)$, $\gamma_t$ has no self-intersections in $B_{2r}(0)\cup B_{2r}(p)$, and the curvature of $\gamma_t$ along with $\frac{x^{\bot}}{|x|^2}$ and all its derivatives are bounded (independently of $t$) in $B_{2r}(0)\cup B_{2r}(p)$.
\item[A2)] Away from the origin and for all $0\leq t<T$, the curves $\gamma_t$ solve the equation 
$$\frac{dx}{dt}=\vec k-\frac{x^{\bot}}{|x|^2}.$$
\end{itemize}
A simple modification  of \cite[Theorem 1.3]{angenent} implies that, for $t>0$, the self-intersections of $\gamma_t$ are finite and non increasing with time.

\begin{thm}\label{angenent}
There is a continuous curve $\gamma_T$ and a finite number of points $\{Q_1,...,Q_m\}\subseteq \C\setminus B_{2r}(0)\cup B_{2r}(p)$ such that $\gamma_T\setminus\{Q_1,...,Q_m\}$ consists of smooth arcs and away from the singular points the curves $\gamma_t$ converge smoothly to $\gamma_T$.  Any two smooth arcs intersect only in finitely many points.

For each of the singular points $Q_i$ and for each small $\varepsilon$,   the number of self-intersections of $\gamma_T$ in $B_{\varepsilon}(Q_i)$ is strictly less than the number of self-intersections of   $\gamma_{t_j}$ in $B_{\varepsilon}(Q_i)$ for some sequence $(t_j)_{j\in \N}$ converging to $T$.
 \end{thm}
 \begin{proof}
 Condition $A1)$ implies that the curves $\gamma_t$ converge smoothly in $ B_{2r}(0)\cup B_{2r}(p)$ as $t$ tends to $T$.  A slight modification of \cite[Theorem 4.1]{angenent0} shows that the quantity
 $$\int_{\gamma_t}|\vec k|d\H^1$$
 is uniformly bounded. Indeed the only change one has to make concerns the existence of boundary terms when  integration by parts is performed. Fortunately, A1) implies that  the contribution from the boundary terms is uniformly bounded and so all the other arguments in \cite[Theorem 4.1]{angenent0} carry through. 
 
 The fact that the total curvature is uniformly bounded and that,  on $\C\setminus B_{2r}(0)$, the deformation vector $\vec k-\frac{x^{\bot}}{|x|^2}$ satisfies  conditions $(V^*_1),$ $(V_2),$ $(V_3),$ $(V^*_5),$ and $(S)$ of \cite{angenent0}, shows that  we can apply \cite[Theorem 5.1]{angenent} to conclude the existence of a continuous curve $\gamma_T$ and a finite number of points $\{Q_1,...,Q_m\}\subseteq \C\setminus (B_{2r}(0)\cup B_{2r}(p))$ such that $\gamma_T\setminus\{Q_1,...,Q_m\}$ consists of smooth arcs and away from the singular points the curves $\gamma_t$ converge smoothly to $\gamma_T$. We note that  \cite[Theorem 5.1]{angenent}  is applied to close curves but an inspection of the proof shows that all the arguments are local and so they apply with no modifications to $\gamma_t$ provided hypothesis A1)  hold.
 
 Oaks \cite[Theorem 6.1]{oaks} showed that for each of the singular points $Q_i$ and for each small $\varepsilon$,   there is a sequence $(t_j)_{j\in \N}$ converging to $T$ so that $\gamma_{t_j}$ has self-intersections in $B_{\varepsilon}(Q_i)$ and either a closed loop of $\gamma_{t_j}$  in $B_{\varepsilon}(Q_i)$ contracts  as $t_j$ tends to $T$ or else there are two distinct arcs in  the smooth part of $\gamma_T$ which coincide in a neighborhood of $Q_i$ (see \cite[Figure 6.2.]{angenent}). Using the fact that the deformation vector is analytic in its arguments on $\C\setminus B_{2r}(0)$, we can argue as in \cite[page 200--201]{angenent} and conclude that the smooth part of $\gamma_T$ must in fact be real analytic in $\C\setminus B_{2r}(0)$. Therefore, any two smooth arcs intersect only in finitely many points and this excludes the second possibility.
 \end{proof}

 \subsection{Non avoidance principle for equivariant flow}
 
\begin{lemm}\label{equi.max}

For each $j=1,2$ consider smooth curves $\sigma_{j,t}:[-a,a]\longrightarrow \C$ defined for all $0\leq t\leq T$ so that
\begin{itemize}
	\item[i)] $\sigma_{j,t}(-s)=-\sigma_{j,t}(s)$ for all $0\leq t\leq T$ and $s\in [-a,a]$.
	\item[ii)] The curves $\gamma_t$ solve the equation 
$$\frac{dx}{dt}=\vec k-\frac{x^{\bot}}{|x|^2}.$$
     \item[iii)] $\sigma_{1,0}\cap\sigma_{2,0}=\{0\}$ (non-tangential intersection) and $(\partial \sigma_{1,t})\cap \sigma_{2,t}=\sigma_{1,t}\cap (\partial \sigma_{2,t})=\emptyset$ for all $0\leq t\leq T$.
\end{itemize}
For all $0\leq t\leq T$ we have $\sigma_{1,t}\cap\sigma_{2,t}=\{0\}$.
\end{lemm}
\begin{proof}
	Away from the origin, it is simple to see  the maximum principle holds and so two disjoint solutions cannot intersect for the first time away from  the origin. Thus it suffices to focus on what happens around the origin. Without loss of generality we assume that $\sigma_{j,t}(s)=(s,f_{j,t}(s))$ for all $s\in[-\delta,\delta]$ for all $t\leq T_1$. The functions $\alpha_{j,t}(s)= s^{-1}f_{j,t}(s)$ are smooth by i) and so we consider $u_t=\alpha_{1,t}-\alpha_{2,t}$ which, form iii),  we can assume to be  initially positive and $u_t(\delta)=u_t(-\delta)>0$ for all $t\leq T_1$. It is enough to show that $u_t$ is positive for all $t\leq T_1$.
We have at once that
	\begin{multline*}
	\frac{df_{j,t}}{dt}=(\arctan(\alpha_{j,t}))'+\frac{f''_{j,t}}{1+(f'_{j,t})^2}\\
	\implies \frac{d\alpha_{j,t}}{dt}=\frac{\alpha''_{j,t}}{1+(s\alpha'_{j,t}+\alpha_{j,t})^2 }+\frac{\alpha'_{j,t}}{s}\frac{1}{1+\alpha_{j,t}^2}+\frac{\alpha'_{j,t}}{s}\frac{2}{1+(s\alpha'_{j,t}+\alpha_{j,t})^2}
	\end{multline*}
The functions $s^{-1}\alpha'_{j,t}$ are smooth for all $s$ and so we obtain
$$  \frac{du_{t}}{dt}=\frac{u''_{t}}{1+C_1^2}+C_2u_t'+C_3u_t+\frac{u'_{t}}{s}C_4^2,$$
where $C_k$ are smooth time dependent bounded functions for  $k=1,\ldots,4$.

Suppose  $T_1$ is the first time at which $u_t$ becomes zero and consider $v_t=u_te^{-Ct}+\varepsilon(t-T_1)$ with $\varepsilon$ small and  $C$ large. The function $v_t$ becomes zero for a first time $t\leq T_1$ at some point $s_0$ for all small positive $\varepsilon$. At that time we have $u''_t(s_0)\geq 0,$ $u'_t(s_0)=0$, and thus, with an obvious abuse of notation,
$$0\geq \frac{dv_{t}}{dt}(s_0)=\varepsilon +\frac{d}{dt}(u_te^{-Ct})(s_0)\geq\varepsilon + \frac{u'_{t}(s_0)}{s_0}C_4^2e^{-Ct}.$$
If $s_0$ is not zero, the last term on the right is zero. If $s_0$ is zero, then the last term on the right is $u''_t(0)C_4^2e^{-Ct}$ which is nonnegative. In  any case we get a contradiction.
\end{proof}
 
\bibliographystyle{amsbook}

\vspace{20mm}

\end{document}